\documentclass[12pt, twoside, reqno]{amsart}
\calclayout
\usepackage[utf8]{inputenc}
\usepackage[T1]{fontenc}
\usepackage{amsmath,amssymb,latexsym,amsthm}
\usepackage{caption}
\usepackage{tikz-cd}
\usepackage{mathtools}
\usepackage[pdftex]{hyperref}
\usepackage[noabbrev,capitalize]{cleveref}
\usepackage{aliascnt}
\usepackage{tikz}
\usetikzlibrary{shapes.misc}
\tikzset{cross/.style={cross out, draw=black, minimum size=2*(#1-\pgflinewidth), inner sep=0pt, outer sep=0pt},
cross/.default={3pt}}
\usepackage[dvipsnames]{xcolor}

\theoremstyle{plain}

\newaliascnt{lemma}{theorem}
\newtheorem{lemma}[lemma]{Lemma}
\aliascntresetthe{lemma}
\newaliascnt{proposition}{theorem}
\newtheorem{proposition}[proposition]{Proposition}
\aliascntresetthe{proposition}
\newaliascnt{corollary}{theorem}
\newtheorem{corollary}[corollary]{Corollary}
\aliascntresetthe{corollary}
\newtheorem*{corollary*}{Corollary}
\newaliascnt{definition}{theorem}
\newtheorem{definition}[definition]{Definition}
\aliascntresetthe{definition}
\newaliascnt{question}{theorem}
\newtheorem{question}[question]{Question}
\aliascntresetthe{question}

\newtheoremstyle{claim}%
  {3pt}%
  {3pt}%
  {}%
  {0pt}%
  {\itshape}%
  {.}%
  {.5em}%
  {\thmname{#1}\thmnumber{ #2}\thmnote{ (#3)}}%
\theoremstyle{claim}
\newaliascnt{remark}{theorem}
\newtheorem{remark}[remark]{Remark}
\aliascntresetthe{remark}

\theoremstyle{plain}
\newtheorem{TheoremA}{Theorem}

\newcommand{\proofend}{\hspace*{\fill}~$\Box$\\}

\def\im{\operatorname {im}}
\def\Iso{\operatorname{Iso}}
\def\GL{\operatorname{GL}}
\def\Hom{\operatorname{Hom}}
\def\Aut{\operatorname{Aut}}

\def\cp{{\mathcal P}}

\def\C{\mathbb{C}}
\def\R{\mathbb{R}}
\def\Z{\mathbb{Z}}

\def\RP{\R P}
\def\CP{\C P}

\def\Diff{\operatorname{Diff}}
\def\SO{\operatorname{SO}}

\def\pp{\partial}
\def\Ch{\rm{Ch}}
\def\sth{\,\vert\,}

\def\pkm{\cp_{k,m}}
\def\pkmch{\cp^{\Ch}_{k,m}}

\def\ni{\noindent}

\newcommand{\sslash}{\mathord{/\mkern-6mu/}}

\def\id{\operatorname{id}}

\def\proof{\noindent {\it Proof. \;}}
\newcommand{\proofof}[1]{\ni {\it Proof of #1. }}

\usetikzlibrary{calc}
\usepackage{enumitem}

\crefname{TheoremA}{Theorem}{Theorems}
\Crefname{TheoremA}{Theorem}{Theorems}
\usetikzlibrary{arrows.meta}

\begin{document}

\title[On Lagrangians of Oakley--Usher type]{On Lagrangians of Oakley--Usher type}

\author{Jo\'e Brendel}  

\address{
D-MATH,
ETH Zürich, 
Rämistrasse 101,
8092 Zürich,
Switzerland }
\email{joe.brendel@math.ethz.ch}

\author{Andrea Piccirilli}

\address{
\parbox[t]{\dimexpr\linewidth-\parindent\relax}{
Institut de math\'ematiques,
Universit\'e de Neuchâtel,
Rue Emile-Argand 11,
\newline
\qquad 2000 \mbox{Neuchâtel}, Switzerland
\vspace{0.2 cm}}}
\email{andrea.piccirilli@unine.ch}

\date{\today}

\begin{abstract}
The main theme of this paper is to study symplectic (un-)knottedness of Lagrangian submanifolds which are not tori. Our construction of such Lagrangians is inspired by the generalized Polterovich Lagrangians of Oakley--Usher and relies on an adaptation of symplectic reduction methods going back to Chekanov--Schlenk and McDuff's probes to the case of non-abelian group actions. We compare Oakley--Usher Lagrangians to their (un-)knotted cousins to find that, depending on the situation, they are:
\begin{enumerate}
    \item not diffeomorphic, 
    \item diffeomorphic, but not Lagrangian isotopic;
    \item Lagrangian isotopic, but not Hamiltonian isotopic,
    \item Hamiltonian isotopic.
\end{enumerate}
This relationship between them frequently changes when the ambient space is compactified by a symplectic cut. In the companion paper~\cite{OU2}, we focus on monotone tori in the three-dimensional quadric and compute their enumerative invariants and study their relationship to mirror symmetry.
\end{abstract}

\maketitle

\section{Introduction}

The study of knottedness and unknottedness of Lagrangian submanifolds is a central theme in symplectic topology. Since the foundational work by Chekanov \cite{Che96} and Eliashberg--Polterovich \cite{EliPol97}, it has been known that there are Lagrangian tori which have the same soft invariants, but are distinct up to symplectomorphisms of the ambient space. On the other hand, there is no local Lagrangian knotting in dimension four, see again \cite{EliPol97}. In the meantime there has been some significant progress: Auroux \cite{Aur15} has proved that there are infinitely many monotone tori in $\R^6$ which are symplectically distinct. Vianna \cite{Via16} has proved that the same holds in some compact four-dimensional symplectic manifolds. In \cite{Bre23b}, it was proved that knottedness of certain small Lagrangian tori occurs in \emph{every} geometrically bounded symplectic manifold of dimension greater than or equal to six. 

Here we study the analogous comparison for families whose members need not be tori. Our starting point is the construction of Oakley--Usher~\cite{OakUsh16}, which generalizes the Polterovich torus studied by Albers--Frauenfelder~\cite{AlbFra08}. The resulting Lagrangians have the topology of products of spheres or discrete quotients of such products. Using Hamiltonian group actions and symplectic reduction, we construct their Chekanov counterparts and determine how the relation between the two families changes under compactification.

\subsection{New Lagrangian submanifolds}

We construct Chekanov counterparts~$\cp_{k,m}^{\Ch}$ of the Oakley--Usher Lagrangians~$\cp_{k,m}$ and compare the two families in~$T^*S^n$,~$T^*\RP^n$, and their symplectic-cut compactifications~$Q^n$ and~$\CP^n$. Throughout,~$0\leq k\leq m$,~$m>0$, and~$n=k+m+1$. The parameters specifying the sizes are omitted when they do not affect the comparison. The first theorem concerns cotangent bundles.

\begin{TheoremA}
    \label{thm:compopen}
    The Lagrangians~$\mathcal{P}_{k,m}, \mathcal{P}^{\Ch}_{k,m} \subset T^*S^n$ are 
    \begin{itemize}
       \item[\textcolor{red}{\labelitemi}] have non-isomorphic integral homology groups if $k\ne0$ and at least one of $k,m$ is even,
        \item[\textcolor{blue}{\labelitemi}] diffeomorphic but not Lagrangian isotopic if either~$k,m$ are both odd and~$(k,m) \neq (1,1)$ or~$k=0$ and~$m \neq 1$;
        \item[\textcolor{ForestGreen}{\labelitemi}] Lagrangian isotopic but not Hamiltonian isotopic if~$(k,m) \in \{(0,1),(1,1)\}$.
    \end{itemize}
    The Lagrangians~$\mathcal{P}_{k,m}, \mathcal{P}^{\Ch}_{k,m} \subset T^*\RP^n$ are 
    \begin{itemize}
        \item[\textcolor{blue}{\labelitemi}] diffeomorphic but not Lagrangian isotopic if~$(k,m) \neq (1,1)$;
       \item[\textcolor{ForestGreen}{\labelitemi}] Lagrangian isotopic but not Hamiltonian isotopic if~$(k,m) =(1,1)$.
    \end{itemize}
    In fact, for $k=0$ in $T^*\RP^n$, we prove that the Lagrangians are not even smoothly isotopic.
    \end{TheoremA}
    
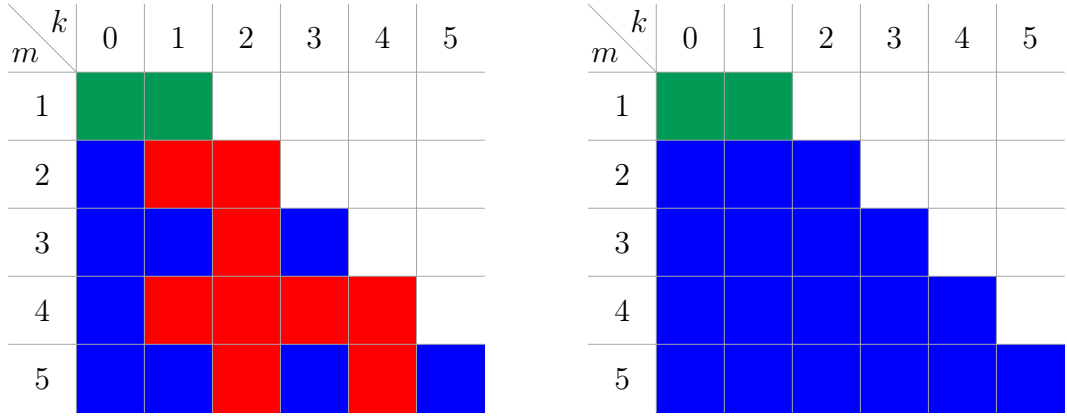
\begin{figure}[h!]
\centering
\begin{tikzpicture}[scale=0.9, line cap=round, line join=round]

\def\n{6}

\fill[ForestGreen](1,1) rectangle (2,5);
\fill[blue](1,1) rectangle (2,4);
\fill[blue](1,1) rectangle (2,3);
\fill[blue](1,1) rectangle (2,2);
\fill[blue](1,0) rectangle (2,1);

\fill[ForestGreen](2,1) rectangle (3,5);
\fill[red](2,1) rectangle (3,4);
\fill[blue](2,1) rectangle (3,3);
\fill[red](2,1) rectangle (3,2);
\fill[blue](2,0) rectangle (3,1);

\fill[red](3,1) rectangle (4,4);
\fill[red](3,1) rectangle (4,3);
\fill[red](3,1) rectangle (4,2);
\fill[red](3,0) rectangle (4,1);

\fill[blue](4,1) rectangle (5,3);
\fill[red](4,1) rectangle (5,2);
\fill[blue](4,0) rectangle (5,1);

\fill[red](5,1) rectangle (6,2);
\fill[red](5,0) rectangle (6,1);

\fill[blue](6,0) rectangle (7,1);

\foreach \x in {1,...,6}
  \draw[gray!70, very thin] (\x,0) -- (\x,6);

\foreach \y in {1,...,5}
  \draw[gray!70, very thin] (0,\y) -- (7,\y);

\draw[gray!70, very thin] (0,6) -- (1,5);
\node at (0.75,5.75) {$k$};
\node at (0.25,5.25) {$m$};

\foreach \c in {0,...,5}
  \node at (\c+1.5,5.5) {\c};

\foreach \r in {1,...,5}
  \node at (0.5,5-\r+0.5) {\r};

\end{tikzpicture}
\hspace{1 cm}
\begin{tikzpicture}[scale=0.9, line cap=round, line join=round]

\def\n{6}

\fill[blue](1,1) rectangle (2,5);
\fill[blue](1,1) rectangle (2,4);
\fill[blue](1,1) rectangle (2,3);
\fill[blue](1,1) rectangle (2,2);
\fill[blue](1,0) rectangle (2,1);
\fill[ForestGreen](2,1) rectangle (3,5);
\fill[blue](2,1) rectangle (3,4);
\fill[blue](2,1) rectangle (3,3);
\fill[blue](2,1) rectangle (3,2);
\fill[blue](2,0) rectangle (3,1);

\fill[blue](3,1) rectangle (4,4);
\fill[blue](3,1) rectangle (4,3);
\fill[blue](3,1) rectangle (4,2);
\fill[blue](3,0) rectangle (4,1);

\fill[blue](4,1) rectangle (5,3);
\fill[blue](4,1) rectangle (5,2);
\fill[blue](4,0) rectangle (5,1);

\fill[blue](5,1) rectangle (6,2);
\fill[blue](5,0) rectangle (6,1);

\fill[blue](6,0) rectangle (7,1);

\foreach \x in {1,...,6}
  \draw[gray!70, very thin] (\x,0) -- (\x,6);

\foreach \y in {1,...,5}
  \draw[gray!70, very thin] (0,\y) -- (7,\y);

\draw[gray!70, very thin] (0,6) -- (1,5);
\node at (0.75,5.75) {$k$};
\node at (0.25,5.25) {$m$};

\foreach \c in {0,...,5}
  \node at (\c+1.5,5.5) {\c};

\foreach \r in {1,...,5}
  \node at (0.5,5-\r+0.5) {\r};

\end{tikzpicture}
\caption{Pictorial representation of the comparison theorem in~$T^*S^n$ on the left and $T^*\R P^n$ on the right for the first few~$k\leq m$.}

\end{figure}

This suggests that for Lagrangians which are not tori, the notion of \emph{knottedness} can take on various different meanings, depending on the situation. 

As in \cite{OakUsh16}, we compactify the above cotangent bundles by performing a symplectic cut to obtain the complex quadric $Q^n$ and complex projective space $\CP^n$, respectively. The Lagrangians from \cref{thm:compopen} all come in a family depending on one real parameter (which we will denote by $r > 0$ and $s>0$, respectively), which should be thought of as \emph{size}. For small enough sizes, the Lagrangians can be considered in the compactifications. 

\begin{TheoremA}
    \label{thm:compclosed}
    For $r,r' \in (0,2\pi)$, the Lagrangians~$\mathcal{P}_{k,m}(r), \mathcal{P}^{\Ch}_{k,m}(r') \subset Q^n$ are 
    \begin{itemize}
        \item have non-isomorphic integral homology groups if $k\ne0$ and at least one of $k,m$ is even,
        \item diffeomorphic but not Lagrangian isotopic if~$m\geq k \geq 3$ are both odd, 
        \item  Lagrangian isotopic if~$(k,m)=(1,1)$ (see \cite{OU2} for a detailed study of these tori),

\item Lagrangian isotopic but not Hamiltonian isotopic if~$k = 0$ and~$r'\neq 2\pi-r$,

        \item Hamiltonian isotopic if~$k = 0$ and~$r'= 2\pi-r$.
    \end{itemize}
    The Lagrangians~$\mathcal{P}_{k,m}(s), \mathcal{P}^{\Ch}_{k,m}(\pi-s) \subset \CP^n$ are Hamiltonian isotopic for all $s \in (0,\pi)$.

\end{TheoremA}

In the case where $k=1$, and $m\geq 3$ is odd, we do not know whether the Lagrangians are Lagrangian isotopic -- they are not distinguished by their Maslov class, see \cref{prop:k1modd}. Observe that the comparison changes drastically when the space is compactified. The same phenomenon can already be observed in the classical example of the Polterovich torus. In $T^*S^2$, it is not Hamiltonian isotopic to the Chekanov torus. After compactifying to $S^2\times S^2$, it is Hamiltonian isotopic to the Chekanov torus of the appropriately chosen area. In particular, there is intersection rigidity between every such Hamiltonian isotopy and the symplectic sphere coming from the compactification. The same holds in the general setting. 

\begin{corollary*}
Write $D_Q=Q^n\cap\{z_0=0\}$ for the cut divisor in $Q^n$, and $D_P=\{[z_0:\cdots:z_n]\in\CP^n:\sum_{j=0}^n z_j^2=0\}$ for the cut divisor in $\CP^n$. Both are copies of the quadric $Q^{n-1}$. Then the trace of any Hamiltonian isotopy mapping~$\mathcal{P}_{k,m}(s)$ to~$\mathcal{P}^{\Ch}_{k,m}(\pi-s)$ in~$\CP^n$ intersects~$D_P$. Similarly, the trace of every Hamiltonian isotopy mapping~$\mathcal{P}_{0,m}(r)$ to~$\mathcal{P}^{\Ch}_{0,m}(2\pi-r)$ in~$Q^n$ intersects~$D_Q$. 
\end{corollary*}

\proof It will follow from our construction that there is a symplectomorphism of the unit cotangent disk bundle~$D^*\RP^n$ to~$\CP^n \setminus D_P$ mapping~$\mathcal{P}_{k,m}(r), \mathcal{P}^{\Ch}_{k,m}(r) \subset D^*\RP^n$ to their respective versions in~$\CP^n$. Now the claim follows from Theorems \ref{thm:compopen} and \ref{thm:compclosed}. Similarly, we have~$Q^n \setminus D_Q \cong D^*S^n$. 
\proofend

\subsection{Construction} The construction is best understood through symplectic reduction. Succinctly, one could say that \emph{$\cp^{\rm Ch}_{k,m}$ is to the conventional Oakley--Usher Lagrangian what the Chekanov torus in $\R^4$ is to the Clifford torus.} We therefore call $\cp_{k,m}^{\rm Ch}$ \emph{Chekanov Oakley--Usher Lagrangian}, or \emph{Chekanov OU-Lagrangian} for short. 

\begin{figure}[ht]
    \centering
    \begin{tikzpicture}
        \fill[opacity=0.2] 
            (0,0) circle (2);
        \node[cross] at (0,0) {};
        \draw[red, thick] (0,0) circle (0.8);
        \node[text=red] at (-1.0,0.9) {\small~$\pkm(r)$};
        \draw[blue, thick] (0,-1.2) arc(-90:90:1.2) arc(270:90:0.2) arc(90:-90:1.6) arc(270:90:0.2);
        \node[text=blue] at (1.7,1.4) {\small~$\pkmch(a)$};
        \node[text=black] at (0,-2.5) {\small~$\R^2\setminus\{0\}$};
    \end{tikzpicture}
    
    \caption{ \label{fig:chek} The symplectic quotient after reduction by $\SO(k+1) \times \SO(m+1)$ is the punctured plane. The red curve lifts to the Oakley--Usher Lagrangian, the blue curve to its Chekanov version.}
\end{figure}
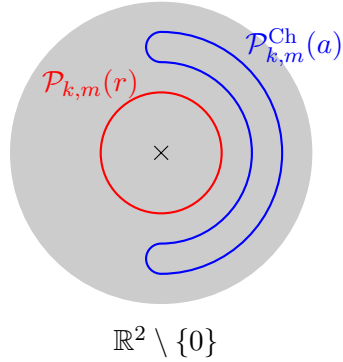

Before discussing the construction, let us briefly recall how to view the Clifford and Chekanov tori from the point of view of symplectic reduction. On $\R^4 = \C^2$, consider the Hamiltonian $S^1$-action given by positive rotation in the first and negative rotation in the second factor in $\C \times \C$. This action is generated by the Hamiltonian $H = \pi\vert z_1\vert^2 - \pi \vert z_2 \vert^2$. The smooth part of the symplectic quotient at the level $H=0$ can be identified with the punctured plane equipped with the standard symplectic form; the puncture corresponds to the fixed point $(0,0) \in \C^2$ of the circle action. Since the Clifford torus is invariant under the action, it projects to a circle in the quotient, and it is easy to check that this circle encloses the puncture. This situation is illustrated in \cref{fig:chek}, where the red circle corresponds to the Clifford torus. Lifting the blue circle on the other hand yields the Chekanov torus in $\C^2$. This idea goes back to \cite{EliPol97}, and was first explicitly formulated using symplectic reduction in \cite{CheSch10}. Very similar ideas have led to the advent of McDuff's so-called \emph{probes}, which can be used to displace toric fibres. 

Our point of view on Oakley--Usher Lagrangians and their Chekanov counterparts is similar to the above discussion, with the crucial difference that the Hamiltonian group action we consider is by a group which is non-abelian in many cases. Let us make this precise on the space $T^*S^n$, the situation is similar in the other ambient spaces. The obvious smooth action of $\SO(n+1)$ on $S^n$ lifts to a Hamiltonian action on its cotangent bundle $T^*S^n$. For any $k,m$ with $n = k + m + 1$, the block-diagonal embedding of $G_{k,m} = \SO(k+1) \times \SO(m+1)$ into $\SO(n+1)$ yields a natural Hamiltonian $G_{k,m}$-action on $T^*S^n$. As was already pointed out in \cite{OakUsh16}, the Oakley--Usher Lagrangian~$\pkm$ is invariant under this action. Performing symplectic reduction with respect to this action on the zero-level set of the corresponding moment map, after removing the zero-section, we obtain a two-dimensional reduced space symplectomorphic to the punctured plane. As for the case of the Clifford torus, the image of $\cp_{k,m}(r)$ is the red curve enclosing the origin. The enclosed symplectic area is~$r$ for~$k=0$ and~$r/2$ for~$k\ne0$, in the normalization of Proposition~\ref{prop:OU_as_lifts}. Lifting the blue curve (which does not enclose the origin!) in the same reduced space defines the Chekanov OU-Lagrangian. 

This can be viewed as a non-abelian version of probes. By this approach, we prove various displaceability results for~$\pkm$ and~$\pkmch$. In particular, we precisely recover the displaceability results proved in \cite{OakUsh16} for~$\pkm$ by different methods.

\subsection{Relation to the companion paper}
This paper develops the construction and the comparison results in arbitrary dimension. The companion paper~\cite{OU2} is a detailed study of the monotone tori~$\cp_{1,1}$ and~$\cp_{1,1}^{\Ch}$ in~$Q^3$. Their distinction from one another and from lifted Vianna tori is proved in~\cite[Theorem A]{OU2}; the superpotential, non-displaceability and Fukaya-category generation statements for~$\cp_{1,1}$ appear in~\cite[Theorem~B]{OU2}. Those enumerative arguments are not needed for the construction and comparison results proved here.

\subsection{Outline of the paper}
Section~\ref{sec:preliminaries} recalls Maslov classes, monotonicity and the topology of mapping tori. Section~\ref{sec:construction} constructs the Oakley--Usher Lagrangians and their Chekanov counterparts using symplectic reduction, and discusses symplectic cuts and the projective-space quotient. Section~\ref{sec:comparison} proves Theorems~\ref{thm:compopen} and~\ref{thm:compclosed}.

\subsection*{Notation and conventions}
We identify~$S^1 = \{e^{2\pi i \theta} \in \C \sth \theta \in \R\} = \R/\Z$ throughout the paper. We use $\iota_{X_H}\omega=-dH$. On cotangent bundles, $\lambda=p\,dq$ and $\omega=d\lambda=\sum dp_j\wedge dq_j$; on $\C$, the form is $dx\wedge dy$. Our Fubini--Study normalization is $\int_{\CP^1}\omega_{FS}=2\pi$. When speaking of homology groups of a manifold~$X$, we assume~$\Z$-coefficients unless specified, and write~$H_1(X)$ in place of~$H_1(X;\Z)$. By~$X \sslash_{\mu = c}G$ we denote the symplectic quotient obtained from performing symplectic reduction at the level~$c$ on a symplectic manifold~$X$ by a Hamiltonian~$G$-action generated by a moment map~$\mu$. By~$\overline{X}_{H \leqslant 1}$, we denote the space obtained as the lower half from performing a symplectic cut at the level~$H = 1$ of a Hamiltonian~$H$ generating a Hamiltonian~$S^1$-action. Throughout the paper, we tacitly make the following fairly natural abuse of notation. Let~$F,H$ be two commuting moment maps on the same space~$X$. If~$X$ admits symplectic reduction at the level~$F=c$, then the group action generated by~$H$ naturally descends to the symplectic quotient. We denote the Hamiltonian generating it again by~$H$.

\subsection*{Acknowledgments} 
This paper and its companion~\cite{OU2} grew out of a Master's thesis written by AP under the supervision of JB at ETH Zürich. We thank Paul Biran for orchestrating this. We warmly thank Leonid Polterovich for drawing our attention to Oakley--Usher Lagrangians, and Johanna Bimmermann, Marc Fares, Johannes Hauber, David Keren Yaar, Joel Schmitz, Mike Usher, and Renato Vianna for interesting discussions, comments and questions. We are particularly grateful to Felix Schlenk for asking many interesting questions, which led to the study of tori presented in Section~2 in~\cite{OU2}. JB is supported by SNSF Ambizione Grant PZ00P2-223460.

\subsection*{Statement on the use of LLMs} An earlier version of this project comprising this paper and its companion \cite{OU2} was written completely without the use of large language models. In subsequent versions, GPT was used to help split the project into two papers, and proofread the text. The authors take full responsibility for the final mathematical content.

\section{Topological preliminaries}
\label{sec:preliminaries}

In this section we discuss some background material, which will be used in later sections: Maslov index computations, monotonicity, and the topology of certain mapping tori.

\subsection{Maslov class and monotonicity}
 \label{algebra section}
Let~$(M^{2n},\omega)$ be a symplectic manifold, and~$L \subset M$ be a Lagrangian submanifold. One of our main tools for distinguishing Lagrangians is the minimal Maslov number of~$L$, defined below. It is invariant under Lagrangian isotopy. Recall that not being Lagrangian isotopic is strictly stronger than not being Hamiltonian isotopic. One of the main methods in this paper is to explicitly compute the
minimal Maslov numbers for the Lagrangians of interest. Furthermore, to prove monotonicity of the Lagrangian submanifolds considered, we will need to study the set~$\omega(\pi_2(M)) \subset \R$.

For these purposes, we will encounter exact sequences of the form
\[\begin{tikzcd}
 \pi_2(L) \arrow[r, "i^2_*"] & \pi_2(M) \arrow[r, "j_*"] & \pi_2(M,L) \arrow[r, "\partial"] & \pi_{1}(L)\arrow[r, "i^1_*"] & \pi_1(M) \arrow[r] & \pi_1(M)/\text{im}(i^1_*),
  \end{tikzcd}
\]
obtained from the homotopy long exact sequence of the pair~$(M,L)$. Note that~$\pi_1(M)/\text{im}(i^1_*)$ is a set equipped with a~$\pi_1(M)$-action but not a group in general, as~$\im(i^1_*)$ need not be a normal subgroup of~$\pi_1(M)$.

The Maslov class is a homomorphism~$m_L \in \text{Hom}(\pi_2(M,L),\Z)$, defined as follows. Let~$[u] \in \pi_2(M,L)$. After choosing a symplectic trivialization of~$u^*TM$ , we obtain a loop~$u|_{\partial D}^*TL$ in the Lagrangian Grassmannian~$\Lambda(n)$, representing a homotopy class in~$\pi_1(\Lambda(n)) \cong \Z$. We assign the integer (winding number) associated to this homotopy class, called the \textit{Maslov index} of~$[u]$, as the value of~$m_L([u])$. 
The area class~$A_L \in \Hom(\pi_2(M,L),\R)$ is simply defined as 
\begin{equation}
A_L([u]):= \int_{D} u^* \omega.
\end{equation}
Define the minimal Maslov number by $m_L(\pi_2(M,L))=N_L\Z$, taking $N_L=0$ when the image is zero.
\begin{definition}
Let~$(M,\omega)$ be a symplectic manifold. A Lagrangian~$L\subset M$ is called monotone if there exists~$\lambda > 0$ such that~$A_L = \lambda m_L$.
\end{definition}

We now explain how the Maslov and area classes can be
described in terms of the fundamental group of $L$.
Let $i:L\hookrightarrow M$ denote the inclusion, and set
\[
G_L:=\ker(i_*^1:\pi_1(L)\to\pi_1(M)).
\]
In particular, if $M$ is simply connected, then
$G_L=\pi_1(L)$. By exactness of the homotopy long sequence, of the pair $(M,L)$,
we have
\[
\pi_2(M,L)/\operatorname{im}(j_*)\cong G_L.
\]
Recall that for every $B\in\pi_2(M)$,
\[
m_L(j_*(B))=2\langle c_1(M),B\rangle,
\qquad
A_L(j_*(B))=\langle[\omega],B\rangle.
\]
Let $N_M\geq 0$ be the minimal Chern number of $M$,
defined by $c_1(\pi_2(M))=N_M\Z$.
Since the Maslov class takes values in $2N_M\Z$
on $\operatorname{im}(j_*)$, it induces a unique
homomorphism
\[
\overline m_L:G_L\longrightarrow\Z/2N_M\Z
\]
such that
\[
\operatorname{pr}\circ m_L
=\overline m_L\circ\partial,
\]
where $\operatorname{pr}:\Z\to\Z/2N_M\Z$
denotes the canonical projection.

Analogously, the area class induces a homomorphism
\[
\overline A_L:
G_L\longrightarrow\R/\omega(\pi_2(M)).
\]
The induced Maslov class determines the minimal Maslov
number $N_L$. Indeed, the subgroup
$m_L(\pi_2(M,L))\subset\Z$ is the inverse image of
$\overline m_L(G_L)$ under the canonical projection
$\Z\to\Z/2N_M\Z$.
Its nonnegative generator is precisely $N_L$.

We now distinguish the ambient spaces considered
in this paper.

\begin{enumerate}

\item
For $M=T^*S^n$ and $M=T^*\RP^n$, with $n\geq 2$,
both the symplectic area and the first Chern class
vanish on $\pi_2(M)$. Therefore, the induced maps take
values in $\Z$ and $\R$, respectively:
\[
\overline m_L:G_L\to\Z,
\qquad
\overline A_L:G_L\to\R.
\]
Since $T^*S^n$ is simply connected, we have $G_L=\pi_1(L)$. For $T^*\RP^n$, we have
\[
G_L=\ker\bigl(
\pi_1(L)\to\pi_1(T^*\RP^n)\cong\Z_2
\bigr).
\]
For $n=2$, this statement is still valid, even if $\pi_2(T^*S^2)$ and $\pi_2(T^*\RP^2)$ do not vanish.

\item
For $M=\CP^n$ or $M=Q^n$, the ambient space is
simply connected, and hence the induced Maslov class
is defined on the whole fundamental group:
\[
\overline m_L:
\pi_1(L)\longrightarrow\Z/2N_M\Z.
\]
Similarly, we have
\[
\overline A_L:
\pi_1(L)\longrightarrow
\R/\omega(\pi_2(M)).
\]
Since $\pi_1(L)$ is the domain of the induced Maslov
class in this case, it can also be seen as an
element of $H^1(L;\Z/2N_M\Z)$.

\end{enumerate}

Finally, suppose that $M$ is monotone, i.e.
there exists $\kappa>0$ such that
\[
[\omega]|_{\pi_2(M)}
=\kappa\,c_1(M)|_{\pi_2(M)}.
\]
Then $\omega(\pi_2(M))=\kappa N_M\Z$. If $N_M>0$ and $L$ is monotone, its monotonicity
constant is necessarily $\lambda=\kappa/2$.
In this case the induced area and Maslov classes
satisfy
\[
\overline A_L
=\Phi_{\kappa/2}\circ\overline m_L,
\]
where
\[
\Phi_{\kappa/2}:
\Z/2N_M\Z\longrightarrow\R/\kappa N_M\Z
\]
is induced by multiplication by $\kappa/2$.

\subsection{Mapping tori}
\label{mapping tori}
Let~$G$ be an Abelian topological group, and~$X$ be a CW-complex. We define 
\[
\textnormal{Bun}_G(X) := \{\text{isomorphism classes of~$G$-principal bundles over~$X$}\}.
\]
It is a standard fact in algebraic topology that there is a bijection between~$\textnormal{Bun}_G(X)$ and~$[X,BG]$, the set of homotopy classes of maps from~$X$ to the classifying space of~$G$.

In the case where $X=S^1$, the classification can be expressed
in terms of the topology of $G$. In general, since homotopy
classes are unbased, we have
\[
\operatorname{Bun}_G(S^1)
\cong [S^1,BG]
\cong \operatorname{Conj}(\pi_1(BG))
\cong \operatorname{Conj}(\pi_0(G)),
\]
where $\operatorname{Conj}(H)$ denotes the set of conjugacy
classes of a group $H$. The last identification follows from
the homotopy equivalence $\Omega BG\simeq G$.
If $G$ is Abelian, conjugacy is trivial, and we obtain
\[
\operatorname{Bun}_G(S^1)\cong\pi_0(G),
\]
see \cite{HusemollerFB}. Now let~$k,m\in\mathbb{N}\cup \{0\}$. Recall that every fiber bundle 
\begin{equation}
\label{firbre bundles}
    S^k \times S^m \xhookrightarrow{} E \to S^1,
\end{equation}
can be obtained from a unique (up to bundle isomorphism)~$\text{Diff}(S^k \times S^m)$-principal bundle~$P \to S^1$ via the Borel construction (see again \cite{HusemollerFB}):
\begin{equation}
E\cong P\times_{\Diff(S^k\times S^m)}(S^k\times S^m).
\end{equation}
Let~$\Gamma:= \langle a_k,a_m \rangle \subset \text{Diff}(S^k \times S^m)$, where~$a_k(x,y) := (-x,y)$ and analogously for~$a_m$. As a topological group,~$\Gamma \cong \Z_2 \times \Z_2$ is discrete. Hence, 
\[
\textnormal{Bun}_{\Gamma}(S^1) \cong [S^1, B\Gamma] \cong \pi_0(\Gamma) \cong \Gamma.
\]
The inclusion $\Gamma \xhookrightarrow{} \text{Diff}(S^k \times S^m)$ induces a map 
\[
\textnormal{Bun}_{\Gamma}(S^1) \to \textnormal{Bun}_{\text{Diff}(S^k \times S^m)}(S^1)
\]
and we are interested in its image. Since the antipodal map $a_n : S^n \to S^n$ is isotopic to the identity if and only if $n$ is odd, 
we obtain 
\[
\im(\Gamma \to \pi_0(\text{Diff}(S^k \times S^m))) \cong
\begin{cases}
1, \qquad \qquad\:\:\: \; k,m \text{ odd},\\
\Z_2, \qquad \:\qquad\text{exactly one of $k,m$ even},\\
\Z_2 \times \Z_2, \qquad k,m \text{ even}.
\end{cases}
\]
Notice that if two such bundles are isomorphic as fibre bundles over~$S^1$, then their total spaces are diffeomorphic as smooth manifolds. Hence, for bundles with structure group reduced to~$\Gamma$, the above classification determines the diffeomorphism type of the total space.
For~$\varphi\in \Diff(F)$, its \textit{mapping torus} is
\[
T_\varphi:=\frac{[0,1]\times F}{(1,z)\sim (0,\varphi(z))}.
\]
Thus, the four bundles above have total spaces
\[
T_{\id}\cong S^1\times S^k\times S^m, \qquad T_{a_k}\cong  (S^1\times S^k)/\Z_2\times S^m,
\]
\[
T_{a_m}\cong S^k\times  (S^1\times S^m)/\Z_2, \qquad  T_{a_ka_m}\cong (S^1\times S^k\times S^m)/\Z_2.
\]
For even~$k$, the antipodal mapping torus~$(S^1\times S^k)/\Z_2$ is a nontrivial topological~$S^k$-bundle over~$S^1$. For odd~$k$, the antipodal map is isotopic to the identity, so this mapping torus is the trivial bundle; a bundle with orientation-reversing monodromy is nontrivial. In the next proposition only, let~$T_{a_k}$ denote the mapping torus of the antipodal map~$a_k:S^k\to S^k$, without the additional~$S^m$-factor. Besides being an~$S^k$-bundle over~$S^1$, it is an~$S^1$-bundle over~$\RP^k$.

\begin{proposition}
$T_{a_k}$ admits a natural structure of~\(S^1\)-bundle over~\(\R P^k\), and in fact it can be written as the unit circle bundle 
\[
T_{a_k}\cong S(\gamma^1\oplus \gamma^1),
\]
where~\(\gamma^1\to \R P^k\) is the tautological line bundle.
\end{proposition}

\proof
Let~\(p:S^k\to \R P^k\) be the double cover, viewed as a principal~\(\mathbb Z_2\)-bundle, where~\(\mathbb Z_2\) acts on~\(S^k\) by the antipodal involution. By definition, the tautological line bundle is the associated bundle
\[
\gamma^1 \cong S^k\times_{\mathbb Z_2}\mathbb R,
\]
where~\(\mathbb Z_2\) acts on~\(\mathbb R\) by the sign representation. Hence
\[
\gamma^1\oplus \gamma^1 \cong S^k\times_{\mathbb Z_2}(\mathbb R\oplus \mathbb R),
\]
where $\Z_2$ acts by multiplication by $-1$ on both summands. Passing to unit sphere bundles gives
\[
S(\gamma^1\oplus \gamma^1)\cong S^k\times_{\mathbb Z_2}S^1,
\]
where~\(\mathbb Z_2\) acts on~\(S^1\subset \mathbb R\oplus \mathbb R\) by the involution induced from~\((u,v)\mapsto(-u,-v)\). This involution is the rotation by $\pi$, hence is orientation-preserving and isotopic to the identity on $S^1$. But this quotient is precisely the mapping torus
\[
\frac{[0,1]\times S^k}{(1,x)\sim(0,-x)}=T_{a_k},
\]
proving the claim.
\proofend

As observed above, the manifolds under consideration are in several cases mapping tori. We will use the following standard exact sequence, called \textit{Wang exact sequence}, to study their homology.
\begin{lemma}[Wang exact sequence]
Let~$f \in \textnormal{Homeo}(F)$ and let 
\begin{equation}
E_f:=\frac{[0,1]\times F}{(1,z)\sim (0,f(z))}
\end{equation}
denote its mapping torus. Then there is a natural long exact sequence in homology:
\begin{equation}\label{eq:Wang}
\cdots \to H_q(F) \xrightarrow{f_*-\mathrm{id}} H_q(F) \to H_q(E_f) \to H_{q-1}(F) \xrightarrow{f_*-\mathrm{id}} H_{q-1}(F) \to \cdots
\end{equation}
\end{lemma}

\begin{proposition}
\label{prop: homotopy distinction}
Let~$E_{f}$ be the mapping torus of~$f$ as above. If there exists~$q \in \mathbb{N}\cup \{0\}$ such that the induced map~$f_*-\mathrm{id}: H_q(F) \to H_q(F)$ is not the zero map, then~$E_f$ is not homologous to~$S^1 \times F$.

\end{proposition}
\proof
Assume that~$E_f$ is homologous to~$S^1 \times F$. From Künneth's formula, we have~$H_*(E_f) \cong H_*(F) \oplus H_{*-1}(F)$. Therefore, for every~$q \in \mathbb{N}\cup \{0\}$, the following short exact sequence holds: 
\[
\begin{tikzcd}
0 \arrow[r] & H_q(F) \arrow[r] & H_q(S^1 \times F) \arrow[r] & H_{q-1}(F) \arrow[r] & 0.
  \end{tikzcd}
\]
From the Wang exact sequence, it follows that~$f_*-\mathrm{id} \equiv 0$ on~$H_*(F)$.
\proofend

\section{Construction}
\label{sec:construction}

In this section, we use symplectic reduction to provide an alternate viewpoint on Oakley--Usher Lagrangians and construct their Chekanov versions. The subsections \S\ref{ssec:OU_reduction} and \S\ref{ssec:OU_quadric} do not contain any essentially new ideas when compared to \cite{OakUsh16}; they merely provide a slightly different point of view.

\subsection{Oakley--Usher Lagrangians and symplectic reduction}

\label{ssec:OU_reduction}

Let~$k,m \in \Z_{\geqslant 0}$, not both zero, and~$n = k + m + 1 \geq 2$. The integers~$k,m$ will be interchangeable and hence we can assume~$k \leqslant m$ and~$m > 0$. One of our main tools to study Oakley--Usher Lagrangians is a Hamiltonian action of~$S^1 \times \SO(k+1) \times \SO(m+1)$ on~$T^*S^n \setminus 0_{S^n}$. We begin with the~\(\SO(k+1) \times \SO(m+1)\)-part, which is induced by the natural $\SO(n+1)$-action. Recall that any diffeomorphism~\(\phi \colon Q \to Q\) induces a symplectomorphism of~\(T^*Q\) via its cotangent lift
\begin{equation}
T^*\phi(q,p)=\bigl(\phi(q),(D_q\phi)^{-T}p\bigr).
\end{equation}
Via this lift, the standard $\SO(n+1)$-action on $S^n$ yields a Hamiltonian action by $\SO(n+1)$ on $T^*S^n$. Viewing
\begin{equation}
T^*S^n=\{(q,p)\in \R^{n+1}\times \R^{n+1}\mid |q|=1,\ q\cdot p=0\},
\end{equation}
this lifted action is especially simple: since each~\(A\in \SO(n+1)\) acts linearly and orthogonally, one has~\(D_qA=A\) and~\(A^{-T}=A\), so
\begin{equation}
(q,p)\longmapsto (Aq,Ap).
\end{equation}
Using the splitting~$\R^{n+1}=\R^{k+1}\times \R^{m+1}$,
we may write~$q=(q_1,q_2)$ and~$p=(p_1,p_2)$, and restricting along the natural block-diagonal inclusion~$\SO(k+1)\times \SO(m+1)\hookrightarrow \SO(n+1)$ yields a Hamiltonian action of~$G_{k,m}:=\SO(k+1)\times \SO(m+1)$ on~$T^*S^n$.

The~$S^1$-part of the action is given by a reparameterization of the geodesic flow which makes it 1-periodic. This reparameterization is generated by the Hamiltonian~$H(q,p) = 2 \pi \vert p \vert$ and does not extend to the zero section. We will refer to this as the \emph{geodesic~$S^1$-action}. Together, this gives an~$S^1\times \SO(k+1)\times \SO(m+1)$-action on~$T^*S^n \setminus 0_{S^n}$ defined by
\begin{equation}
    \label{eq:groupactionmk}
    \begin{split}
        \Phi^{k,m}_{\theta,A,B}(q_1,q_2,p_1,p_2) = \:&\Bigg(\bigg(\cos 2\pi \theta Aq_1+\sin 2\pi\theta \frac{Ap_1}{|p|}, 
        \cos 2\pi\theta Bq_2+\sin 2\pi\theta \frac{Bp_2}{|p|}\bigg),\\
        &\bigg (\cos  2\pi\theta Ap_1-\sin  2\pi\theta |p|Aq_1,\cos 2\pi \theta Bp_2-\sin 2\pi\theta |p|Bq_2\bigg)\Bigg).
    \end{split}
\end{equation}

\begin{remark}
    \label{rk:globalstab}
    This group action is effective if and only if at least one of~$k,m$ is even. If both are odd, the element~$(\frac{1}{2},-\id,-\id)$ acts by the identity.
\end{remark}

Since~$SO(n+1)$ acts on~$S^n$ by isometries with respect to the standard metric, its Hamiltonian lift to~$T^*S^n$ preserves the reparameterized geodesic flow and thus the~$S^1$- and~$G_{k,m}$-actions commute. Taking these facts together, we obtain the following.

\begin{proposition}
The group action \eqref{eq:groupactionmk} is Hamiltonian and generated by the moment map
\begin{equation}
    \label{eq:momentmapmk}
    \begin{split}
    \mu_{k,m} \colon T^*S^n \setminus 0_{S^n} &\rightarrow \R \times \mathfrak{g}_{k,m}^*, \\
    (q,p) = (q_1,q_2,p_1,p_2) & \mapsto (2 \pi \vert p \vert, q_1 \wedge p_1, q_2 \wedge p_2),
    \end{split}
\end{equation}
where~$\mathfrak{g}_{k,m}^* := \textnormal{Lie}(G_{k,m})^*$ is the dual of the Lie algebra, and where we have used the natural identification~$\mathfrak{so}(k+1)^* \cong \Lambda^2 \R^{k+1}$ and similarly for~$m$. 
\end{proposition}

\begin{definition}
    For every~$k,m$ and~$r>0$, we define the \emph{Oakley--Usher Lagrangian}~$\pkm(r) \subset T^*S^n$ as the orbit 
    \begin{equation}
        \pkm(r) := (S^1 \times G_{k,m})\cdot z_r
    \end{equation}
    of the reference point 
    \begin{equation}
    \label{eq:refpoint}
    z_r = \left(e_1,0,0,\frac{r}{2\pi}e_{m+1}\right) \in 
    T^*S^n \subset \R^{n+1} \times \R^{n+1}.
    \end{equation}
    under the group action \eqref{eq:groupactionmk}.
\end{definition}
We shall verify in Proposition \ref{prop:OU_as_lifts} that~$\pkm(r)$ is indeed Lagrangian. It is easy to see that this definition agrees with the one given in \cite[Section 5]{OakUsh16} up to rescaling the \emph{size}~$r > 0$ of the submanifold. From \cite[Proposition 5.1]{OakUsh16} it follows that every~$\pkm(r)$ is a non-displaceable, monotone Lagrangian submanifold. For~$k=0$, we have~$G_{k,m} \cong \SO(m+1)$, since~$\SO(1)$ is the trivial group. In most of the paper, this case needs special attention, as it differs substantially from the case~$k \neq 0$. For~$k=0$, the stabilizer
\begin{equation}
    \label{eq:stabeq0}
    \operatorname{stab}_{z_r}(S^1 \times \SO(m+1)) \cong \SO(m),
\end{equation}
is given by the set of elements fixing~$e_{m+1}$. Since the Oakley--Usher Lagrangian is defined as the orbit of~$z_r$, we find that its diffeomorphism type is given by
\begin{equation}
    \label{eq:OU_top_keq0}
    \cp_{0,m}(r) \cong (S^1 \times \SO(m+1))/\SO(m) \cong S^1 \times S^m.
\end{equation}
For~$k \neq 0$, the situation is slightly more complicated. For every integer~$l \geqslant 2$ let~$R_l \in \SO(l)$ be the diagonal matrix~$R_l = \operatorname{diag}(-1,1,\ldots,1,-1)$ with only~$\pm 1$ entries on the diagonal such that all entries are~$+1$, except for the first and the last one. We find
\begin{equation}
    \label{eq:stabkneq0}
    \operatorname{stab}_{z_r}(S^1 \times G_{k,m}) 
    \cong (\SO(k) \times \SO(m))\rtimes \Z_2,
\end{equation}
where~$\SO(k)$ is given by all elements in~$\SO(k+1)$ fixing~$e_1$ and~$\SO(m)$ is given by all elements in~$\SO(m+1)$ fixing~$e_{m+1}$. Notice however that, as opposed to the case~$k=0$, this does not give the full stabilizer. Indeed, the element~$(\frac{1}{2},R_{k+1},R_{m+1})$ fixes~$z_r$, too. It generates the~$\Z_2$-copy in \eqref{eq:stabkneq0} by acting on~$\SO(k) \times \SO(m)$ by conjugation by~$(R_{k+1},R_{m+1})$. For~$k \neq 0$, we find the diffeomorphism type 
\begin{align}
    \cp_{k,m}(r) 
    & \cong \frac{S^1\times G_{k,m}}{(\SO(k)\times\SO(m))\rtimes\Z_2} \nonumber\\
    & \cong (S^1 \times S^k \times S^m)/\Z_2,
    \label{eq:OU_top_kneq0}
\end{align}
where the last~$\Z_2$-action is by the antipodal map on each of the sphere factors simultaneously. Recall from the discussion in Section \ref{mapping tori} that if~$k,m$ are both odd,~$\cp_{k,m}(r)$ is diffeomorphic to~$S^1 \times S^k \times S^m$ and if only one of them is odd, say~$m$, then it is diffeomorphic to~$(S^1 \times S^k)/\Z_2 \times S^m \cong S(\gamma^1\oplus \gamma^1)\times S^m$. We can give an explicit description of the space in question by considering the map 
\begin{equation*}
    (S^1 \times S^k \times S^m)/\Z_2 \rightarrow S^1, \quad
    [\theta,x,y] \mapsto 2\theta,
\end{equation*}
which exhibits the space in question as a mapping torus of the map~$(x,y) \mapsto (-x,-y)$ on~$S^k \times S^m$.

\begin{proposition}
The Oakley--Usher Lagrangian can be viewed as the level set of the moment map,
\begin{equation}
\label{P as moment fibre}
\pkm(r) = \mu_{k,m}^{-1}(r,0,0).
\end{equation}
\end{proposition}
\proof
Since~$\mu_{k,m}(z_r) = (r,0,0)$, the orbit~$\pkm(r)$ is contained in~$\mu_{k,m}^{-1}( r,0,0)$. Indeed, the moment map is equivariant with respect to the coadjoint action of $S^1 \times G_{k,m}$ on its Lie algebra, and $(r,0,0)$ is fixed by that action. Conversely, if~$(q,p)\in \mu_{k,m}^{-1}(r,0,0)$, then~$q_1\wedge p_1 = 0$ and~$q_2\wedge p_2 = 0$, so~$q_1,p_1$ and~$q_2,p_2$ are collinear. Together with the constraints~$|q| = 1$,~$q\cdot p = 0$ and~$2\pi |p| = r$, this forces
\[
(q,p) = \left(\cos(2\pi \theta) x,\sin(2\pi \theta)y,-\frac{r}{2\pi}\sin(2\pi \theta)x,\frac{r}{2\pi}\cos(2\pi \theta)y\right),
\]
for some~$\theta \in S^1$ and~$(x,y) \in S^k \times S^m$. By transitivity of the $\SO(k+1)\times\SO(m+1)$-action, we can find~$A \in SO(k+1), B\in SO(m+1)$ such that~$Ae_1 = x$ and~$Be_{m+1} = y$. Therefore,~$(q,p) = \Phi^{k,m}_{\theta, A,B}(z_r)$, which proves the claim.
\proofend

To develop our alternative viewpoint on the Oakley--Usher Lagrangians, we would like to perform symplectic reduction with respect to the~$G_{k,m}$-part of the action \eqref{eq:groupactionmk}. To that end, we consider the moment map of that sub-action
\begin{equation}
    \overline{\mu}_{k,m} \colon T^*S^n \rightarrow \mathfrak{so}(k+1)^* \times \mathfrak{so}(m+1)^* = \mathfrak{g}_{k,m}^*.
\end{equation}
This moment map extends smoothly to the zero-section~$0_{S^n}$ and so does the Hamiltonian action it generates. Let us now discuss symplectic reduction on the level~$\overline{\mu}_{k,m} = 0$. This seems hazardous at first, since~$0$ is not a regular value of~$\overline{\mu}_{k,m}$. However, in the spirit of stratified symplectic reduction in the sense of Lerman--Sjamaar \cite{SjaLer91}, we will see that we can still make sense of this. Let~$\phi^H$ be the~$S^1$-action on~$T^*S^n \setminus 0_{S^n}$ generated by~$H(q,p) = 2\pi \vert p \vert$.  Define the map~$I_{k,m}$ by

\begin{equation}
    \label{eq:Ikm}
    I_{k,m} \colon \R_{>0} \times S^1 \times S^k \times S^m \rightarrow T^*S^n, \quad
    (\rho,\alpha,x,y) \mapsto \phi^H_{\alpha}\left(x,0,0,\rho y\right),
\end{equation}
if~$k \neq 0$ and by~$I_{0,m}(\rho,\alpha,y) = \phi^H_{\alpha}(1,0,0,\rho y)$ if~$k=0$.

\begin{proposition}
    \label{prop:Ikm_properties}
    The map~$I_{k,m}$ as defined by \eqref{eq:Ikm} has the following properties:
    \begin{enumerate}
        \item It is~$S^1 \times G_{k,m}$-equivariant with respect to the obvious~$S^1 \times G_{k,m}$-action on the domain and the action \eqref{eq:groupactionmk} on the target;
        \item it maps sets~$\{\rho\} \times S^1 \times S^k \times S^m$ (and~$\{\rho\} \times S^1 \times S^m$ for~$k=0$) to $\pkm(2\pi\rho) = \mu_{k,m}^{-1}(2\pi\rho,0,0)$;
        \item it is a diffeomorphism to~$\overline{\mu}_{0,m}^{-1}(0) \setminus 0_{S^n}$ when~$k=0$;
        \item it is a smooth double cover of~$\overline{\mu}_{k,m}^{-1}(0) \setminus 0_{S^n}$ when~$k \neq 0$;
        \item it satisfies
            \begin{equation*}
                I_{k,m}^* \omega = 2\pi \,d\rho \wedge d\alpha,
            \end{equation*}
        where~$\omega = d\lambda$ denotes the canonical symplectic form on~$T^*S^n$.
    \end{enumerate}
\end{proposition}

\proof Properties \emph{(1)} and \emph{(2)} are immediate. To prove \emph{(3)} and \emph{(4)}, note that the level set in question can be decomposed as 
\begin{equation*}
    \overline{\mu}^{-1}_{k,m}(0)\setminus 0_{S^n}
    = \bigcup_{\rho >0} \mu^{-1}_{k,m}(\rho ,0,0)
    = \bigcup_{\rho >0} \pkm(\rho).
\end{equation*}
The claims then follow from equivariance, and \eqref{eq:stabeq0} and \eqref{eq:stabkneq0}, respectively. For Property \emph{(5)}, we will compute the pullback~$I^*_{k,m}\omega$ on the tangent space
\begin{equation}
    T_{(\rho,0,x_0,y_0)}(\R_{>0} \times S^1 \times S^k \times S^m)
    = T_\rho \R_{>0} \times T_0 S^1 \times T_{x_0}S^k \times T_{y_0}S^m.
\end{equation}
Note that it is sufficient to consider the case~$\alpha = 0$, since the symplectic form is invariant under the~$S^1$-action. In the above product decomposition of the tangent space, take vectors~$\pp_{\rho} \in T_{\rho} \R_{>0}, \pp_{\alpha} \in T_0 S^1, X, X' \in T_{x_0}S^k, Y, Y' \in T_{y_0}S^m$, and compute
\begin{align*}
    (I_{k,m})_* \pp_{\rho} &= (0,0,0,y_{0}), \quad
    (I_{k,m})_* X^{(\prime)} = (X^{(\prime)},0,0,0), \\
    (I_{k,m})_* \pp_{\alpha} &= 2\pi(0,y_0,-\rho x_0,0), \quad
    (I_{k,m})_* Y^{(\prime)} = (0,0,0,\rho Y^{(\prime)}).
\end{align*}
Recall that the canonical symplectic form~$\omega$ on~$T^*S^n$ is given by the restriction of~$\sum_{i=1}^{n+1} dp_i \wedge dq_i$ to~$T^*S^n \subset \R^{n+1} \times \R^{n+1}$. Therefore, 
\begin{align*}
   & I_{k,m}^* \omega (\pp_{\rho},X) = I_{k,m}^* \omega (\pp_{\rho},Y) = 0, \\
   & I_{k,m}^* \omega (X,X') = I_{k,m}^* \omega (Y,Y') = I_{k,m}^* \omega (X,Y) = 0, \\
   & I_{k,m}^* \omega (\pp_{\alpha},X) = -2\pi\rho x_0 \cdot X = 0, \\
   & I_{k,m}^* \omega (\pp_{\alpha},Y) = -2\pi\rho y_0 \cdot Y = 0.
\end{align*}
The only non-zero term is~$I_{k,m}^* \omega (\pp_{\rho},\pp_{\alpha}) = 2\pi y_0 \cdot y_0 = 2\pi$, proving the claim.
\proofend

\cref{prop:Ikm_properties} allows us to set up symplectic reduction away from the zero-section, despite the fact that~$0$ is not a regular value of~$\overline{\mu}_{k,m}$. We start with the (easier) case of~$k=0$. The quotient space of~$\R_{>0} \times S^1 \times S^m$ under the~$\SO(m+1)$-action is simply~$\R_{>0} \times S^1$. Using properties \emph{(1)} and \emph{(3)}, we find that this induces a quotient map 
\begin{equation}
    \label{eq:def_symp_red}
    \pi \colon \overline{\mu}_{0,m}^{-1}(0) \setminus 0_{S^n} \rightarrow \R_{>0} \times S^1,
\end{equation}
which, by property \emph{(5)} satisfies~$\pi^*(d\rho \wedge d\alpha) = \omega$. In other words, the symplectic manifold~$(\R_{>0} \times S^1 , d\rho \wedge d\alpha)$ is the symplectic reduced space of the~$G_{0,m} \cong \SO(m+1)$-action on~$\overline{\mu}^{-1}_{0,m}(0) \setminus 0_{S^n}$. 

For~$k \neq 0$, the double cover~$I_{k,m}$ induces a diffeomorphism~$\R_{>0} \times (S^1 \times S^k \times S^m)/\Z_2 \rightarrow \overline{\mu}_{k,m}^{-1}(0) \setminus 0_{S^n}$, where~$\Z_2$ acts by the antipodal map on all of the sphere factors simultaneously. Indeed, this follows from the equivariance of~$I_{k,m}$ together with \eqref{eq:stabkneq0}. The quotient map~$\R_{>0} \times (S^1 \times S^k \times S^m)/\Z_2 \rightarrow \R_{>0} \times \RP^1$ is defined as~$[\rho,\alpha,x,y] \mapsto [\rho,\alpha]$, meaning that the \emph{symplectic quotient} is~$\R_{>0} \times \RP^1$ equipped with the form~$d\rho \wedge d\alpha$ in this case. Note however that the natural Hamiltonian~$(\rho,\alpha) \mapsto \rho$ is~$1/2$-periodic on~$\R_{>0} \times \RP^1$. For convenience, we thus identify~$\R_{>0} \times \RP^1$ with~$\R_{>0} \times S^1$ via the symplectomorphism~$(\rho,\alpha) \mapsto (\rho/2,2\alpha)$. On this new space, the Hamiltonian~$(\rho,\alpha) \mapsto \rho$ induces an honest~$S^1$-action as per our convention~$S^1 = \R / \Z$. Thus we find a quotient map as in \eqref{eq:def_symp_red} for~$k \neq 0$.

\begin{definition}
    \label{def:reductionTSn}
    We call 
    \begin{equation*}
        Y := (\R_{>0} \times S^1 , d\rho \wedge d\alpha)
    \end{equation*}
    the \emph{symplectic quotient} of~$(T^*S^n)^{\times} = T^*S^n \setminus 0_{S^n}$ by the~$G_{k,m}$-action. We denote this quotient by~$(T^*S^n)^{\times}\sslash_{\overline{\mu}_{k,m} = 0} G_{k,m}$ and its \emph{quotient map} by
    \begin{equation}
        \label{eq:reductionTSn}
        \pi \colon \overline{\mu}_{k,m}^{-1}(0) \setminus 0_{S^n} \rightarrow Y.
    \end{equation}
\end{definition}

\begin{remark}
    \label{rk:singreductionTSn}
    We denote by~$\widehat{Y}$ the \emph{reduced space}~$T^*S^n \sslash_{\overline{\mu}_{k,m} = 0} G_{k,m} = \overline{\mu}_{k,m}^{-1}(0)/G_{k,m}$ (i.e. without removing the zero-section). It is given by adding a distinguished point~$y_*$ to~$Y$ and extending~$\pi$ by continuity, such that~$\pi^{-1}(y_*) = S^n$. The resulting space~$\widehat{Y}$ is a plane containing one distinguished point, which, in the language of Lerman--Sjamaar \cite{SjaLer91}, belongs to a lower stratum of the stratified reduced space. We will usually only consider the smooth part~$Y \subset \widehat{Y}$, but it is sometimes convenient to think in terms of~$\widehat{Y}$ instead of~$Y$.
\end{remark}

From the point of view of the above construction, we can view the Oakley--Usher Lagrangians as lifts of curves
\begin{equation*}
    S^1(r) = \{r\} \times S^1 \subset Y = \R_{>0} \times S^1,\quad
    r > 0. 
\end{equation*}
\begin{proposition} 
\label{prop:OU_as_lifts}
The symplectic quotient defined in \cref{def:reductionTSn} has the following properties:
\begin{enumerate}
\item For~$k=0$, the map \eqref{eq:reductionTSn} is a trivial~$S^m$-bundle over~$Y$.
\item For~$k \neq 0$, the map \eqref{eq:reductionTSn} is an~$S^k \times S^m$-bundle over~$Y$ with monodromy given by the factor-wise antipodal map~$(x,y) \mapsto (-x,-y)$.
\item The Oakley--Usher Lagrangians can be viewed as the following lifts:
    \begin{equation}
        \label{eq:OU_as_lifts_TSn}
        \pkm(r)
        =
        \begin{cases}
            \pi^{-1}(S^1(r)) & k = 0; \\
            \pi^{-1}(S^1(r/2)) & k \neq 0.
        \end{cases}
    \end{equation}
\item More generally, the lift~$\pi^{-1}(C)$ of any embedded smooth curve~$C \subset Y$ is a Lagrangian submanifold of~$T^*S^n$. If two such curves are Hamiltonian isotopic, then so are their Lagrangian lifts.
\end{enumerate}
\end{proposition}
\vspace{0.8 cm}
\begin{figure}[ht]
\centering
\begin{tikzpicture}[scale=1.0, thick, baseline]

\draw[dashed] (0,0) ellipse (0.5 and 1.5);

\draw (0,1.5) -- (5,1.5);
\draw (0,-1.5) -- (5,-1.5);

\draw[dashed] (5,1.5) -- (7,1.5);
\draw[dashed] (5,-1.5) -- (7,-1.5);

\draw[red, very thick] (3,0) ellipse (0.5 and 1.5);
\node[red, above] at (3,1.6) {$S^1(r)= \{r\}\times S^1$};

\node at (8.5,0) {$r \to +\infty$};

\node at (3,-2.5) {$Y = \mathbb{R}_{>0} \times S^1$};
\end{tikzpicture}
\caption{The reduced space $Y$, topologically a cylinder, and a curve $S^1(r)$ at fixed height.}
\end{figure}
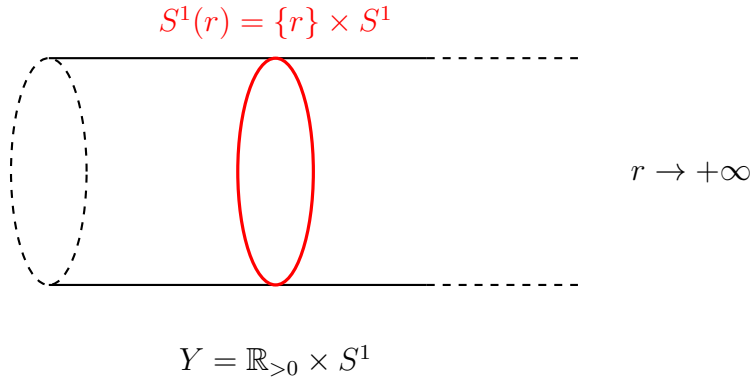

\proof Point \emph{(3)} follows immediately from \emph{(2)} in \cref{prop:Ikm_properties} and the construction of the symplectic reduction. The~$1/2$-factor in the second case comes from post-composing the natural quotient map with~$(\rho,\alpha) \mapsto (\rho/2,2\alpha)$ in the case~$k \neq 0$. Points \emph{(1)} and \emph{(2)} are a consequence of point \emph{(3)} together with the topology of the Oakley--Usher Lagrangians, as discussed in \eqref{eq:OU_top_keq0} and \eqref{eq:OU_top_kneq0}. The lift in point \emph{(4)} is of the correct dimension and the fact that the symplectic form vanishes on its tangent space follows from the fact that the restriction of the symplectic form~$\omega$ to~$\overline{\mu}_{k,m}^{-1}(0) \setminus 0_{S^n}$ is given by~$\pi^* (d\rho \wedge d\alpha)$, and any curve in~$Y$ is automatically Lagrangian. For the Hamiltonian isotopies, it suffices to lift the function $h_t$ on the base to a function $h_t \circ \pi$ on the level set and cut it off. The corresponding Hamiltonian flow lifts the Hamiltonian generated by $h_t$ downstairs.
\proofend

\subsection{OU-Lagrangians in the quadric}
\label{ssec:OU_quadric} We view the complex~$n$-dimensional quadric as
\begin{equation*}
    Q^n = \{[z_0 \colon \ldots \colon z_{n+1}] \in \CP^{n+1} \sth z_0^2 + \ldots + z_{n+1}^2 = 0\},
\end{equation*}
which we equip with the symplectic form induced by the Fubini--Study form on~$\CP^{n+1}$.

\begin{remark}
Here we chose a specific model of the quadric, but the resulting symplectic manifold does not depend on that choice. Indeed, smooth degree-$d$ hypersurfaces in the same projective space, equipped with the same Fubini--Study normalization, are symplectomorphic: connect their defining polynomials through the complement of the discriminant, use the resulting smooth family, and apply Moser's argument to its constant cohomology class.
\end{remark}

The following is widely known, but we were unable to find a proof in general dimension. See also the work by Bimmermann~\cite{Bim24} for related results proved by different methods.

\begin{proposition} 
    \label{prop:cut}
    The space~$\overline{T^*S^n}_{H \leqslant 2\pi}$ obtained by performing a symplectic cut at the level~$H = 2\pi \vert p \vert = 2\pi$ is symplectomorphic to the quadric~$Q^n = \{z_0^2 + \ldots + z_{n+1}^2 = 0\} \subset \CP^{n+1}$ equipped with the symplectic form induced by the Fubini--Study form on~$\CP^{n+1}$. 
\end{proposition}

\proof Put $s=|p|$ and $D=Q^n\cap\{z_0=0\}$. The map
\[
    \psi(q,p)=[\sqrt{1-s^2}:p+iq],\qquad s<1,
\]
is a diffeomorphism onto $Q^n\setminus D$. Indeed, in the affine chart
$[1:x+iy]$, its inverse is $(q,p)=(y/|y|,x/|y|)$.
Its lift $Z=(\sqrt{1-s^2},p+iq)$ to the sphere $|Z|^2=2$ satisfies
\[
    Z^*\!\left(\frac12\operatorname{Im}\sum_j\overline Z_j\,dZ_j\right)
    =\frac12(p\cdot dq-q\cdot dp)=p\cdot dq.
\]
The Hopf quotient therefore gives $\psi^*\omega_{\mathrm{FS}}=d(p\cdot dq)$,
with the stated normalization that a projective line has area $2\pi$. We now use the symplectic cut construction. On
$(T^*S^n\setminus\{p=0\})\times\mathbb C$, the diagonal action given by
the geodesic flow and complex rotation has moment map
\[
    F(q,p,w)=2\pi s+\pi|w|^2.
\]
The cut space away from the zero section is $F^{-1}(2\pi)/S^1$,
and its natural symplectic embedding is
\[
    \iota(q,p)=[q,p,\sqrt{2(1-s)}],\qquad 0<s<1.
\]
The full cut space includes the zero section through the unchanged
open subset $D^*S^n$.
On $F^{-1}(2\pi)$, put $u=p/s$ and $\zeta=w/\sqrt{2(1+s)}$ and define
\[
    \overline\psi([q,p,w])
    =[2\zeta:u+iq-\zeta^2(u-iq)].
\]
This lies in $Q^n$ since $(u\pm iq)\cdot(u\pm iq)=0$ and
$(u+iq)\cdot(u-iq)=2$.
Under the circle action, $\zeta$ and $u+iq$ acquire the factor
$e^{i\theta}$, while $u-iq$ acquires $e^{-i\theta}$; thus the formula
descends smoothly to the symplectic quotient. Since
$\zeta^2=(1-s)/(1+s)$ on the image of $\iota$, we have
$\overline\psi\circ\iota=\psi$. Thus $\psi$ also extends this map smoothly
over the zero section.
At $w=0$, it identifies the geodesic quotient with $D$ by
$[q,p]\mapsto[0:u+iq]$: an isotropic complex line determines precisely
an oriented real two-plane, hence an oriented great circle.
Consequently $\overline\psi$ is a smooth bijection onto $Q^n$.
Its pullback of $\omega_{\mathrm{FS}}$ equals the cut form on the dense
interior, hence everywhere. Nondegeneracy makes it a local
diffeomorphism, so it is a symplectomorphism.\proofend

From the proof, we obtain the natural embedding of the unit co-disk bundle into the quadric,
\begin{equation}
    \label{eq:psi_emb}
    \psi \colon D^*S^n \hookrightarrow Q^n.
\end{equation}

\begin{definition}
    For all~$r \in (0,2\pi)$, the image~$\psi(\pkm(r)) \subset Q^n$ of~$\pkm(r) \subset D^*S^n$ is, by abuse of terminology and notation, again denoted by~$\pkm(r) \subset Q^n$ and called the \emph{Oakley--Usher Lagrangian}. 
\end{definition}

Note that the symplectic cut in \cref{prop:cut} is carried out with respect to the first component of the original moment map~$\mu_{k,m}$ on~$T^*S^n$. Therefore, the resulting space~$Q^n$ again carries an~$S^1 \times G_{k,m}$-action on the complement of the image of $0_{S^n}$. Continuing our natural abuse of notation, we again denote the corresponding moment maps by~$\mu_{k,m}$ and~$\overline{\mu}_{k,m}$.
 For all~$r \in (0,2\pi)$, the new Oakley--Usher Lagrangians~$\pkm(r) \subset Q^n$ can again be viewed as the orbit of the group action and a level set of the moment map, respectively:
\begin{equation*}
    \pkm(r) 
    = (S^1 \times G_{k,m})\cdot \psi(z_r)
    = \mu^{-1}_{k,m}(r, 0,0).
\end{equation*}
From now on, we also denote the image of the zero-section~$0_{S^n} \subset T^*S^n$ under the embedding~$\psi$ and its complement by 
\begin{equation}
    \label{eq:sigma_notation}
    \Sigma = \psi(0_{S^n}) \subset Q^n, \quad
    (Q^n)^{\times} = Q^n \setminus \Sigma.
\end{equation}

\begin{remark} 
Here we chose a specific model of the quadric. It is natural to wonder whether our construction depends on this choice. As it is proven in the next proposition, the answer turns out to be negative.
\end{remark}

Since the~$S^1$- and~$G_{k,m}$-actions commute, taking the symplectic cut as in \cref{prop:cut} commutes with the symplectic reduction as in \cref{def:reductionTSn}.

\begin{proposition}
    \label{prop:redcut}
    The spaces~$\overline{((T^*S^n)^{\times}\sslash_{\overline{\mu}_{k,m} = 0}G_{k,m})}_{H \leqslant 2\pi}$ and~$(\overline{(T^*S^n)^{\times}}_{H \leq 2\pi})\sslash_{\overline{\mu}_{k,m} = 0}G_{k,m}$ are symplectomorphic.
\end{proposition}

Here we have abused notation in a fairly natural way: the Hamiltonian~$H$ induces another Hamiltonian~$S^1$-action on the reduced space~$(T^*S^n)^{\times}\sslash_{\overline{\mu}_{k,m} = 0}G_{k,m}$, whose Hamiltonian we again denote by~$H$. It is with respect to the latter that we perform a symplectic cut to get the space~$\overline{((T^*S^n)^{\times}\sslash_{\overline{\mu}_{k,m} = 0}G_{k,m})}_{H \leqslant 2\pi}$. Similarly, in the second expression~$\overline{\mu}_{k,m}$ denotes the moment map induced on the cut space.

\smallskip
\proofof{\cref{prop:redcut}} Recall from \cite{Ler95} that any symplectic cut space~$\overline{X}_{G \leqslant 1}$ can itself be viewed as the symplectic reduced space~$(X \times \C)\sslash_{\widetilde{G} = 1}$, where~$\widetilde{G}(x,w) = G(x) + \pi \vert w \vert^2$. We find
\begin{equation}
    (\overline{(T^*S^n)^{\times}}_{H \leqslant 2\pi})\sslash_{\overline{\mu}_{k,m} = 0}G_{k,m}
    =
    \big(((T^*S^n)^{\times} \times \C)\sslash_{\widetilde{H} = 2\pi}S^1\big) \sslash_{\overline{\mu}_{k,m}} G_{k,m},
\end{equation}
and 
\begin{equation}
    \overline{((T^*S^n)^{\times}\sslash_{\overline{\mu}_{k,m} = 0}G_{k,m})}_{H \leqslant 2\pi}
    =
    \big(((T^*S^n)^{\times} \sslash_{\overline{\mu}_{k,m}} G_{k,m})\times \C\big)\sslash_{\widetilde{H} = 2\pi}\,S^1,
\end{equation}
Since~$G_{k,m}$ acts trivially on the copy of~$\C$, the latter space is symplectomorphic to~$(((T^*S^n)^{\times} \times \C) \sslash_{\overline{\mu}_{k,m}} G_{k,m})\sslash_{\widetilde{H} = 2\pi}S^1$. Since the~$S^1$- and~$G_{k,m}$-actions on~$(T^*S^n)^{\times} \times \C$ commute, we can use reduction by stages to deduce that both spaces are symplectomorphic to~$(T^*S^n)^{\times} \times \C \sslash_{(\widetilde{H},\overline{\mu}_{k,m}) = (2\pi,0)}$, which proves the claim.
\proofend

This allows us to identify the reduced space of the quadric~$(Q^n)^{\times}\sslash_{\overline{\mu}_{k,m} = 0}G_{k,m}$.

\begin{corollary}
\label{cor:redspace}
    If~$k=0$, then the reduced space is an open disk of area~$2\pi$ equipped with the standard symplectic form. If~$k \neq 0$, then the reduced space is an open orbifold-disk of area~$\pi$ with a~$\Z_2$-orbifold point at the origin and equipped with the standard symplectic form on the smooth part.
\end{corollary}

\proof
    By \cref{prop:cut} and \cref{prop:redcut}, we obtain 
    \begin{equation}    
        (Q^n)^{\times}\sslash_{\overline{\mu}_{k,m} = 0}G_{k,m}
        \cong \overline{((T^*S^n)^{\times}\sslash_{\overline{\mu}_{k,m} = 0}G_{k,m})}_{H \leqslant 2\pi}.
    \end{equation}
    The reduced space is~$Y = (T^*S^n)^{\times}\sslash_{\overline{\mu}_{k,m} = 0}G_{k,m}$, see \cref{def:reductionTSn}. For~$k=0$, the Hamiltonian~$H = 2\pi \vert p \vert$ descends to the Hamiltonian~$(\rho,\alpha) \mapsto \rho$ on~$Y$. The cut space $\overline{Y}_{\rho = 2\pi}$ is given by the open disk of area~$2\pi$ as claimed. For~$k \neq 0$, however, the action induced by~$H = 2\pi \vert p \vert$ on the reduced space has~$\Z_2$-stabilizer (it is generated by~$(\rho,\alpha) \mapsto 2\rho$ in our choice of coordinates in~$Y$), see also the construction preceding \cref{def:reductionTSn}. This means that the cut locus is a~$\Z_2$-orbifold point, as claimed.
\proofend

As in \cref{rk:singreductionTSn}, we will often think of the reduced spaces of~$T^*S^n$ and~$Q^n$ by the~$G_{k,m}$-actions without removing the zero-section~$0_{S^n} \subset T^*S^n$ or its image~$\Sigma \subset Q^n$, respectively. In that case, we add a \emph{distinguished point} to the reduced space. By abuse of terminology, we still call the so-obtained space \emph{reduced space}. Though we only use the properties of symplectic reduction on the smooth part of the so-obtained reduced spaces, this is often helpful to visualize the situation. Throughout the rest of the paper, we will use the following notation.

\begin{definition}
    \label{def:YZZ}
    We denote the reduced spaces of~$T^*S^n$ by
    \begin{equation*}
        \widehat{Y} = T^*S^n\sslash_{\overline{\mu}_{k,m} = 0}G_{k,m}, 
        \quad
        Y = (T^*S^n)^{\times}\sslash_{\overline{\mu}_{k,m} = 0}G_{k,m}
    \end{equation*}
    and the reduced spaces of~$Q^n$ when~$k \neq 0$ by
    \begin{equation*}
        \widehat{Z} = Q^n\sslash_{\overline{\mu}_{k,m} = 0}G_{k,m}, 
        \quad
        Z = (Q^n)^{\times} \setminus \{H = 2\pi\} \sslash_{\overline{\mu}_{k,m} = 0}G_{k,m}.
    \end{equation*}
    In the case~$k=0$, we denote the reduced spaces of~$Q^n$ by 
    \begin{equation*}
        \widehat{Z'} = Q^n\sslash_{\overline{\mu}_{0,m} = 0}G_{0,m}, 
        \quad
        Z' = (Q^n)^{\times}\sslash_{\overline{\mu}_{0,m} = 0}G_{0,m}.
    \end{equation*}
\end{definition}

The spaces~$Y,Z,Z'$ are the respective smooth parts of the reduced spaces~$\widehat{Y}, \widehat{Z}, \widehat{Z'}$ which each contain one distinguished point. By the discussion surrounding \cref{def:reductionTSn} and by \cref{cor:redspace} we find that 
\begin{enumerate}
\item~$\widehat{Y}$ is homeomorphic to the plane and contains one distinguished point (corresponding to~$0_{S^n} \subset T^*S^n$) whose complement~$Y$ is symplectomorphic to the half-cylinder~$\R_{>0} \times S^1 = \{(\rho,\alpha)\}$ equipped with the standard form~$d\rho \wedge d\alpha$;
\item~$\widehat{Z}$ is homeomorphic to the two-sphere and contains one distinguished point (corresponding to~$\Sigma \subset Q^n$) and a~$\Z_2$-orbifold point. The complement~$Z$ of these two points is symplectomorphic to the finite cylinder~$((0,\pi) \times S^1, d\rho \wedge d\alpha)$ of area~$\pi$;
\item~$\widehat{Z'}$ is homeomorphic to the two-sphere and contains one distinguished point (corresponding to~$\Sigma \subset Q^n$) whose complement~$Z'$ is symplectomorphic to the standard disk~$D = \{\pi \vert z \vert^2 < 2\pi\}$ of area~$2\pi$.
\end{enumerate}

\begin{proposition}
    \label{prop:OU_quadric_props}
    The Oakley--Usher Lagrangians in~$Q^n$ have the following properties. The second and third points were proved in Propositions 6.3 and 6.5 in \cite{OakUsh16}. For the third point, we present an alternative proof.
    \begin{enumerate}
        \item The Oakley--Usher Lagrangians can be viewed as the following lifts:
            \begin{equation}
                \label{eq:OU_as_lifts_Qn}
                    \pkm(r)
                    =
                    \begin{cases}
                        \pi^{-1}(S^1(r)) & k = 0; \\
                        \pi^{-1}(S^1(r/2)) & k \neq 0,
                    \end{cases}
            \end{equation}
            where~$S^1(r)$ is the circle at height~$\rho = r$ in the respective reduced spaces~$Z,Z'$ we have discussed above;
        \item The OU-Lagrangian~$\cp_{k,m}(r)\subset Q^n$ is monotone if and only if~$r= r_* = \frac{2\pi(n-1)}{n}$;
        \item Any Oakley--Usher Lagrangian~$\cp_{0,m}(r) \subset Q^n$ with~$r > \pi$ is displaceable. In particular, for all~$m \geqslant 2$, the monotone Oakley--Usher Lagrangian~$\cp_{0,m}(r_*) \subset Q^n$ is displaceable.
    \end{enumerate}
\end{proposition}

\proof For point \emph{(1)}, use point \emph{(3)} in \cref{prop:OU_as_lifts} to consider 
\begin{equation*}
    \cp^{Q^n}_{k,m}(r)
    = \psi(\cp^{T^*S^n}_{k,m}(r))
    = \psi(\pi^{-1}(S^1(s))),
\end{equation*}
where~$s = r$ if~$k=0$ and~$s = r/2$ if~$k \neq 0$. Recall from \cref{prop:redcut} that taking the symplectic cut at~$H = 2\pi$ on~$T^*S^n$ commutes with symplectic reduction on~$\overline{\mu}_{k,m}^{-1}(0)$. Therefore, we find that~$\cp^{Q^n}_{k,m} = \pi^{-1}(S^1(s))$, where~$S^1(s) \subset Z,Z'$ is the image of~$S^1(s) \subset Y$ after the symplectic cut in the quotient, as claimed. Point \emph{(2)} is \cite[Proposition 6.3]{OakUsh16}. Point \emph{(3)} is a consequence of point \emph{(1)} together with the fact that~$Z'$ is symplectomorphic to a disk of area~$2\pi$. Indeed, the circle~$S^1(r) \subset Z'$ bounds a smooth disk of area~$2\pi - r$ (since the symplectic cut from~$Y$ to~$Z'$ happens at~$\rho = 2\pi$) and any Hamiltonian displacing the circle~$S^1(r)$ lifts to a Hamiltonian displacing~$\cp_{k,m}(r) \subset Q^n$. For area reasons, such a Hamiltonian exists whenever~$2\pi - r < \pi$. \proofend

\subsection{Chekanov Oakley--Usher Lagrangians} Recall from \cref{prop:OU_as_lifts} and \cref{prop:OU_quadric_props} that the Oakley--Usher Lagrangians in~$T^*S^n$ and~$Q^n$ can be defined as lifts of curves in certain reduced spaces. Often the smooth part of these reduced spaces is symplectomorphic to a cylinder. Thus besides lifting non-contractible curves (yielding Oakley--Usher Lagrangians), we can lift contractible curves, too. As discussed in \cref{prop:OU_as_lifts}, such lifts yield Lagrangian submanifolds in~$T^*S^n$, and the case of the quadric follows by the same token. Again, we start by discussing the case of~$T^*S^n$.

\begin{definition}
    For given~$k,m$, and~$a > 0$, let~$\Gamma(a) \subset Y$ be a closed embedded curve bounding a smooth disk of area~$a$ if~$k = 0$ and of area~$a/2$ if~$k \neq 0$. Define the \emph{Chekanov OU-Lagrangian in~$T^*S^n$} to be~$\cp^{\Ch}_{k,m}(a) = \pi^{-1}(\Gamma(a))\subset T^*S^n$. 
\end{definition}

In the special case~$n=2$, recall that the conventional Oakley--Usher Lagrangian~$\cp_{0,1}(r) \subset T^*S^2$ is simply the Polterovich torus~$T_{\rm Pol}(r)$. The corresponding Chekanov Oakley--Usher Lagrangian, on the other hand, is Hamiltonian isotopic to the torus obtained from embedding a Chekanov torus~$T_{\rm Ch} \subset \C^2$ into~$T^*S^2$ by a symplectic ball embedding. This is a monotone, displaceable Lagrangian. We show that this is true for all~$k,m$.

\begin{proposition}
\label{prop:COUproperties}
The Chekanov Oakley--Usher Lagrangian~$\cp^{\Ch}_{k,m}(a)$ is 
\begin{enumerate}
    \item diffeomorphic to~$S^1 \times S^k \times S^m$ for~$k\neq 0$ and to~$S^1 \times S^m$ for~$k = 0$;
    \item Lagrangian;
    \item well-defined up to Hamiltonian isotopy;
    \item displaceable;
    \item monotone with monotonicity constant~$\lambda = \frac{a}{4}$ for~$k\neq 0$ and~$\lambda = \frac{a}{2}$ for~$k= 0$.
\end{enumerate} 
\end{proposition}
\proof
As~$\Gamma(a)$ is contractible, the fiber bundle~$F_{k,m}\xhookrightarrow{} \overline{\mu}^{-1}_{k,m}(0)\setminus 0_{S^n} \to S^1\times \R_{>0}$ restricted to~$\Gamma(a)$ is a trivial bundle, or~$\cp^{\Ch}_{k,m}(a) = \pi^{-1}(\Gamma(a)) \cong \Gamma(a) \times F_{k,m}$, where~$F_{k,m}$ denotes the fibre of the corresponding fibre bundle. As~$\Gamma(a) \cong S^1$, this proves point \emph{(1)}. Point \emph{(2)} follows from \emph{(4)} in \cref{prop:OU_as_lifts}. To prove well-definedness, we need to prove that, up to Hamiltonian isotopy,~$\cp^{\Ch}_{k,m}(a)$ does not depend on the choice of the curve~$\Gamma(a)$. Since any two such curves are Hamiltonian isotopic, a lift of the Hamiltonian isotopy to~$T^*S^n$ proves~\emph{(3)}. Similarly, any such curve~$\Gamma(a)$ is displaceable and the lift of any displacing Hamiltonian isotopy displaces~$\cp^{\Ch}_{k,m}(a)$, proving \emph{(4)}.

Let us now turn to the proof of monotonicity. Since the case~$(k,m) = (0,1)$ corresponds to the known Chekanov torus in~$T^*S^2$, we restrict our attention to the case~$n = k + m + 1 \geqslant 3$. Thus, we have~$\pi_1(T^*S^n) = 0 = \pi_2(T^*S^n)$ and we are in the first case of the enumeration at the end of \S\ref{algebra section}: the Maslov class factors through a homomorphism~$\overline{m}_{\cp^{\Ch}_{k,m}(a)}: \pi_1(\cp^{\Ch}_{k,m}(a)) \to \Z$ such that~$m_{\cp^{\Ch}_{k,m}(a)}=\overline{m}_{\cp^{\Ch}_{k,m}(a)} \circ \partial$.
Note that~$\pi_1(\cp^{\Ch}_{k,m}(a)) \cong \pi_1(S^1) \oplus \pi_1(F_{k,m})$, as opposed to the case of~$\cp_{k,m}$ treated in \cite[Proposition 5.1]{OakUsh16}, where, for $k>0$, $\pi_1(S^1)\oplus\pi_1(F_{k,m})$ embeds as an index-two subgroup of $\pi_1(\pkm(r))$.
Here we prove \emph{(5)} for the general case~$k\neq 0$. The proof for the case~$k=0$ works similarly; the difference in the monotonicity constant comes from the discrepancy of areas bounded by circles which lift to Oakley--Usher Lagrangians, see point \emph{(3)} in \cref{prop:OU_as_lifts}. 
We consider the composition 
\[
\pi_* \circ \pp \colon \pi_2(T^*S^n, \cp^{\Ch}_{k,m}(a)) \rightarrow \pi_1(\Gamma(a))
\]
with the projection of the symplectic reduction~$\pi$. The kernel of this map is~$\ker (\pi_* \circ \pp) = \pi_1(F_{k,m})$ which is isomorphic to~$\Z^2$ if~$k=m=1$, to~$\Z$ if either~$k=1$, or~$m=1$, or to the trivial group if~$k,m > 1$.

We first show that both the area and the Maslov class
vanish on this kernel. Let $\gamma$ be a loop contained
in a fibre $F=\pi^{-1}(y)$, and let $u$ be any disk
in $T^*S^n$ with boundary $\gamma$. Since $F$ is a $G_{k,m}$-orbit contained in the
zero level of the moment map, the canonical one-form
$\lambda$ vanishes on $F$. Thus, by Stokes' theorem,
\[
A_L([u])=\int_{\gamma}\lambda=0.
\]
To compute the Maslov index, consider the
OU Lagrangian corresponding to the height
circle in $Y$ passing through $y$. This Lagrangian
contains the same fibre $F$. Along $\gamma$, the
tangent spaces of the two Lagrangians contain
$\ker(d\pi)$, and are obtained by taking the inverse
images under $d\pi$ of two lines in $T_yY$.
Joining these lines by a path gives a homotopy of
the corresponding Lagrangian tangent spaces along
$\gamma$. Consequently, their Maslov indices on the
same disk $u$ agree.

Since the OU Lagrangian is
monotone by \cite[Proposition 5.1]{OakUsh16}, and
$A_L([u])=0$, we conclude that $m_L([u])=0$.
Therefore, both classes vanish on
$\ker(\pi_*\circ\partial)$.

Let now ~$\delta \colon (D^2,\pp D^2) \rightarrow (Y,\Gamma(a))$ be the disk such that~$\pp [\delta] \in \pi_1(\Gamma(a)) \cong \Z$ equals the preferred generator. Since~$D^2$ is contractible, the restriction of~$ \pi \colon \overline{\mu}^{-1}_{k,m}(0) \setminus 0_{S^n} \rightarrow Y$ to~$\im \delta$ is a trivial~$F_{k,m}$-bundle. Therefore,~$\delta$ lifts to a map~$\widehat{\delta} \colon (D^2,\pp D^2) \rightarrow (\overline{\mu}^{-1}_{k,m}(0) \setminus 0_{S^n},\cp^{\Ch}_{k,m}(a))$ (the fact that its image is disjoint from the zero-section follows from the fact that the image of $\delta$ lies in the smooth part $Y$ of the reduced space) satisfying
\begin{itemize}
    \item[(i)]~$\pi \circ \widehat{\delta} = \delta$;
    \item[(ii)]~$[\widehat{\delta}] \in \pi_2(T^*S^n, \cp^{\Ch}_{k,m}(a))$ generates~$\pi_2(T^*S^n, \cp^{\Ch}_{k,m}(a))/\ker(\pi_* \circ \pp)$;
    \item[(iii)]~$\int_{\widehat{\delta}} \omega = \int_{\delta} d\rho \wedge d\alpha$;
    \item[(iv)]~$m_{\cp^{\Ch}_{k,m}(a)}[\widehat{\delta}] = m_{\Gamma(a)}[\delta]$.
\end{itemize}
Point (ii) follows from the above. Indeed, we have shown that~$\pi_* \circ \pp$ descends to an isomorphism~$\pi_2(T^*S^n, \cp^{\Ch}_{k,m}(a))/\ker(\pi_* \circ \pp) \rightarrow \pi_1(\Gamma(a))$ and~$\pi_* \circ \pp$ maps~$\widehat{\delta}$ to the generator of~$\pi_1(\Gamma(a))$ by construction. Point (iii) follows from the fact that~$(Y,d\rho \wedge d\alpha)$ is the reduced space~$(T^*S^n)^{\times}\sslash_{\overline{\mu}_{k,m} = 0}G_{k,m}$ and hence~$\omega \vert_{\overline{\mu}^{-1}_{k,m}(0) \setminus 0_{S^n}} = \pi^*(d\rho \wedge d\alpha)$. Point (iv) follows from \cref{lem:maslov_reduction}. Note that the arguments in the proof of this lemma carry over to our case, where we do reduction on~$\overline{\mu}^{-1}_{k,m}(0) \setminus 0_{S^n}$ instead of on the full level set. 

Since~$m_{\Gamma(a)}[\delta] = 2$ and~$\int_{\delta} d\rho \wedge d\alpha = \frac{a}{2}$, the last claim \emph{(5)} of the proposition follows.
\proofend

We now state the required lemma. 

\begin{lemma}
\label{lem:maslov_reduction}
Let~$(X,\omega)$ be a symplectic manifold equipped with a Hamiltonian~$G$-action generated by the moment map~$\mu$ such that~$X$ admits symplectic~$G$-reduction at the level~$\mu = 0$. Denote its quotient by~$X_0 = X\sslash_{\mu = 0}G$ and its reduction projection by~$\pi \colon \mu^{-1}(0) \rightarrow X_0$. Furthermore, let~$L \subset \mu^{-1}(0)$ be a Lagrangian submanifold and~$u \colon (D^2,\pp D^2) \rightarrow (\mu^{-1}(0),L)$ be a disk with boundary on~$L$ contained in~$\mu^{-1}(0)$. Then its Maslov index is 
\begin{equation*}
    m_{L}[u] = m_{\pi(L)}[\pi \circ u],
\end{equation*}
where we recall that~$\pi(L) \subset X_0$ is automatically a smooth Lagrangian.

\end{lemma}
\proof
The proof follows from the one of \cite[Proposition 3.2]{Smi21}. Recall that we do not exactly do symplectic reduction because our spaces are singular at the zero section. However, the Lagrangians considered do not intersect the singular locus, and thus the same reasoning applies.
\proofend

Let us now turn to the construction of the Chekanov Oakley--Usher Lagrangian in the quadric~$Q^n = \overline{T^*S^n}_{H \leqslant 2 \pi}$, which we view again as the space obtained after performing a symplectic cut on the cotangent bundle~$T^*S^n$, see \cref{prop:cut}. Recall from the discussion surrounding \cref{def:YZZ} that the reduced space of~$Q^n$ at the level~$\overline{\mu}_{k,m} = 0$
\begin{enumerate}
    \item is a sphere~$\widehat{Z}'$ of symplectic area~$2\pi$ with one distinguished point in case~$k=0$,
    \item is a sphere~$\widehat{Z}$ of symplectic area~$\pi$ with one distinguished point and one~$\Z_2$-orbifold point in case~$k \neq 0$.
\end{enumerate}

For all~$k$, the curve~$\Gamma(a) \subset Y$ (bounding a disk of area~$a$ if~$k=0$ and area~$a/2$ if~$k \neq 0$) can be made to sit in the respective cut spaces~$Z,Z'$ whenever~$a < 2\pi$.

\begin{definition}
    For all~$a \in (0,2\pi)$, we call~$\pkmch(a) = \pi^{-1}(\Gamma(a)) \subset Q^n$ the \emph{Chekanov OU-Lagrangian in the quadric}. 
\end{definition}

As in the case of the conventional Oakley--Usher Lagrangian, \cref{prop:redcut} shows that this is equivalent to embedding a copy of~$\pkmch(a) \subset D^*S^n$ into~$Q^n$ via the natural symplectic embedding~$\psi$, as defined in \eqref{eq:psi_emb}. The definition using reduction shows furthermore that we can pick a copy of~$\pkmch(a)$ to lie in~$D^*S^n$ whenever~$a \in (0,2\pi)$. 

By the same arguments as above, we find that the Chekanov OU-Lagrangians in the quadric are Lagrangian submanifolds which are well-defined up to Hamiltonian isotopy. Let us now turn to the question of when they are monotone. See \cref{prop:OU_quadric_props} for the analogous statement about the conventional OU-Lagrangians in the quadric. 

Let~$l = [\CP^1] \in H_2(Q^n)$ be the class of the line. A possible representative is the image of the map 
\begin{equation*}
\CP^1 \to Q^n, \qquad [s:t] \to [s:is:t:it:0:...:0].
\end{equation*}
Its symplectic area is~$\int_{\CP^1} \omega = 2\pi$ and its Chern class is~$c_1(Q^n) \cdot l = n$. For~$n\geq3$, the line class~$l$ generates~$H_2(Q^n)$. For~$n=2$, the two ruling classes generate~$H_2(Q^2)$, and each has area~$2\pi$ and Chern number~$2$. In either case the quadric is monotone, with
\begin{equation}
    \label{eq:quadric_monotonicity}
    [\omega] = \frac{2\pi}{n}c_1(Q^n).
\end{equation}

\begin{proposition}
    The Chekanov OU-Lagrangian~$\pkmch(a) \subset Q^n$ is monotone if and only if~$a = a_* = \frac{4\pi}{n}$ for~$k\neq 0$ and~$a = a_* = \frac{2\pi}{n}$ for~$k=0$.
\end{proposition}

\proof Set $L=\cp^{\Ch}_{k,m}(a)$. In the cotangent disk bundle, choose the section disk $\delta$ above, and disks bounding a generating set of the fibre subgroup. The latter have area and Maslov index zero, while
\[
(A(\delta),m(\delta))=
\begin{cases}(a,2),&k=0,\\(a/2,2),&k>0.\end{cases}
\]
Their boundary classes generate $\pi_1(L)$. By the relative exact sequence, these disks together with the image of $\pi_2(Q^n)$ generate the relative group. Each line class has area $2\pi$ and Maslov index $2n$. For $Q^2$, use both ruling classes, each with area $2\pi$ and Maslov index $4$. Therefore we have monotonicity if and only if $A(\delta)=\frac{\pi}{n} m(\delta)$. This yields $a=2\pi/n$ for $k=0$ and $a=4\pi/n$ for $k>0$. For the other generators, the equality trivially holds.
\proofend

By a proof analogous to the one of \cref{prop:OU_quadric_props}, we can prove the following corollary. Note that in the case of the Chekanov OU-Lagrangians, their reduced curves~$\Gamma(a)$ always bound a disk contained in the smooth part of the reduced space, which is why we obtain displaceability results even for~$k \neq 0$. 

\begin{corollary}
    \label{cor:COU_displaceable}
    For~$k = 0$, the Chekanov OU-Lagrangian~$\cp^{\Ch}_{0,m}(a) \subset Q^n$ is displaceable whenever~$a < \pi$. In particular, the monotone copy~$\cp^{\Ch}_{0,m}(a_*) \subset Q^n$ is displaceable for all~$n \geqslant 3$. For~$k \neq 0$, the Chekanov OU-Lagrangian~$\pkmch(a) \subset Q^n$ is displaceable whenever~$a < \pi$. In particular, the monotone copy~$\cp^{\Ch}_{k,m}(a_*) \subset Q^n$ is displaceable for all~$n \geqslant 5$.
\end{corollary}

We will show in \cref{thm:compclosed} that~$\cp_{0,m}^{\Ch}(a)$ is Hamiltonian isotopic to~$\cp_{0,m}(2\pi - a)$, which reduces the first part of the above statement to the third statement of \cref{prop:OU_quadric_props}.

\subsection{A~\texorpdfstring{$\Z_2$}{Z2}-quotient}

 Similar to the treatment in \cite{OakUsh16}, an analogous construction can be carried out in~$T^*\RP^n$ and its compactification via symplectic cut~$\C P^n = \overline{T^*\R P^n}_{\vert p \vert \leqslant 1}$. Up to a few crucial differences, on which we will put particular emphasis, the case of~$T^*\R P^n$ is analogous to the case of~$T^*S^n$, just as the case of~$\C P^n$ is analogous to the one of~$Q^n$, which is why we leave the details and proofs to the reader. By abuse of notation, we use the same symbols for the moment maps, submanifolds and other objects involved.

We view the cotangent bundle~$T^*\R P^n$ as the discrete quotient~$T^*S^n/I$, where~$I$ denotes the involution~$I(q,p) = (-q,-p)$. Denote the quotient map by~$\Pi \colon T^*S^n \rightarrow T^*\R P^n$. Since~$I = \Phi^{k,m}_{1/2,\id,\id}$ for the group action defined in~\eqref{eq:groupactionmk} and since the element~$(\frac{1}{2},\id,\id)$ commutes with all elements in~$S^1 \times G_{k,m}$, we obtain an induced $S^1\times G_{k,m}$-action on $T^*\RP^n\setminus0_{\RP^n}$. Obviously, the element~$(\frac{1}{2},\id,\id)$ acts by the identity on~$T^*\R P^n$. Therefore, we quotient out~$S^1 \times G_{k,m}$ by~$(\frac{1}{2},\id,\id)$, yielding another action of the same group generated by the moment map~$\mu'_{k,m}(\Pi(q,p)) = (\pi \vert p \vert, q_1 \wedge p_1, q_2 \wedge p_2)$. Note that we have changed the Hamiltonian in the first component from $2\pi \vert p \vert$ to $\pi \vert p \vert$, in order for it to define an $S^1=\R/\Z$-action on $T^*\RP^n$. We denote its projection to the latter two components generating the~$G_{k,m}$-action by~$\overline{\mu}'_{k,m}$. Similarly to \cref{prop:cut}, the symplectic cut space~$\overline{T^*\R P^n}_{\vert p \vert \leq 1}$ is symplectomorphic to~$\C P^n$. Denote the so-obtained symplectic embedding by~$\psi \colon D^*\RP^n \hookrightarrow \C P^n$.

Whenever there is a danger of confusing (Chekanov) OU-Lagrangians in $T^*S^n$ with OU-Lagrangians in $T^*\RP^n$ we denote the former by $\cp^{S}, \cp^{{\rm Ch},S}$ and the latter by $\cp^{\RP}, \cp^{{\rm Ch},\RP}$.

\begin{definition}
For $s> 0$, the Oakley--Usher Lagrangian is defined by
\[
\cp^{\RP}_{k,m}(s)=\Pi(\cp^{S}_{k,m}(2s))
 =(\mu'_{k,m})^{-1}(s,0,0).
\]
\end{definition}

These were first defined by Oakley--Usher. They are again orbits of the~$S^1 \times G_{k,m}$-action on~$T^*\RP^n$ and~$\CP^n$. They have similar properties as their counterparts in~$T^*S^n$ and~$Q^n$. In particular, the submanifolds~$\pkm(r) \subset T^*\RP^n$ are monotone, non-displaceable Lagrangians. The copy~$\pkm(s) \subset \CP^n$ is monotone if and only if~$s=s_*:=\frac{\pi(n-1)}{n+1}$. The computation is the same as for the quadric, as again~$l=[\C P^1]$ is a generator of~$H_2(\C P^n)$ with symplectic area~$2\pi$, with the difference that the first Chern class satisfies~$c_1(\C P^n) \cdot l = n+1$.

For $k \neq 0$, we could define the Chekanov Oakley--Usher Lagrangian in $T^*\RP^n$ as the projection of its version in $T^*S^n$. However in the case $k = 0$, the resulting subset may fail to be embedded. Therefore, we rely on the approach via symplectic reduction to define it. For~$k \neq 0$, the stabilizer of the~$S^1 \times G_{k,m}$-action on~$\pkm(r) \subset T^*\RP^n$ is isomorphic to~$\Z_2 \ltimes (\SO(k) \times \SO(m))$, where~$\SO(k)$ fixes~$e_1$, $\SO(m)$ fixes~$e_{n+1}$ and the~$\Z_2$-copy is generated by~$(0,C,D)$ where~$C,D$ are linear transformations satisfying~$Ce_1 = -e_1$ and $De_{n+1} = -e_{n+1}$. Compare this to the situation in~$T^*S^n$ described around \eqref{eq:stabkneq0}. It follows that~$\pkm(r)$ is diffeomorphic to~$\RP^1 \times  (S^k \times S^m)/\Z_2 \cong S^1 \times (S^k \times S^m)/\Z_2$, where~$\Z_2$ acts by the antipodal map in both factors. 

For~$k = 0$, the stabilizer on~$\cp_{0,m}(r)$ is isomorphic to~$\SO(m)$ and~$\cp_{0,m}(r)$ is diffeomorphic to~$S^1 \times S^m$. Contrary to the case of~$T^*S^n$, the OU-Lagrangian in~$T^*\RP^n$ is always a trivial fiber bundle over~$S^1$. 
Similarly to the spherical case, $\pkm(s)=(\mu_{k,m}')^{-1}(s,0,0)$, which can be used to prove the following.

\begin{proposition}
    The reduced space~$(T^*\R P^n)\sslash_{\overline{\mu}_{k,m}' = 0}G_{k,m}$ is a stratified plane whose smooth part is symplectomorphic to the standard punctured plane~$Y = S^1 \times \R_{>0}$. 
\end{proposition}
Note that since the action of~$I$ is in the circle factor of the group~$S^1 \times G_{k,m}$, it factors through the reduction by the~$G_{k,m}$-component.

\begin{definition}
    \label{def: projective YZ}
    We denote the reduced spaces of~$T^*\R P^n$ by
    \begin{equation*}
        \widehat{Y} = T^*\R P^n\sslash_{\overline{\mu'}_{k,m} = 0}G_{k,m}, 
        \quad
        Y = (T^*\R P^n)^{\times}\sslash_{\overline{\mu'}_{k,m} = 0}G_{k,m}
    \end{equation*}
    and the reduced spaces of~$\C P^n$ by
    \begin{equation*}
        \widehat{Z} = \C P^n\sslash_{\overline{\mu'}_{k,m} = 0}G_{k,m}, 
        \quad
        Z = (\C P^n)^{\times}  \sslash_{\overline{\mu'}_{k,m} = 0}G_{k,m}.
    \end{equation*}
\end{definition}

The spaces~$Y,Z$ are the respective smooth parts of the reduced spaces~$\widehat{Y}, \widehat{Z}$ which each contain one distinguished point. By the discussion surrounding \cref{def:reductionTSn} and by \cref{cor:redspace} we find that 
\begin{enumerate}
\item~$\widehat{Y}$ is homeomorphic to the plane and contains one distinguished point (corresponding to~$0_{\R P^n} \subset T^*\R P^n$) whose complement~$Y$ is symplectomorphic to the half-cylinder~$\R_{>0} \times S^1 = \{(\rho,\alpha)\}$ equipped with the standard form~$d\rho \wedge d\alpha$;
\item~$\widehat{Z}$ is homeomorphic to the two-sphere and contains one distinguished point (corresponding to~$\Sigma \subset \C P^n$) whose complement~$Z$ is symplectomorphic to the standard disk~$D = \{\pi \vert z \vert^2 < \pi\}$ of area~$\pi$.
\end{enumerate}

Denote the symplectic quotient map by 
\begin{equation}
\label{eq:reductionTRPn}
 \pi': (\overline{\mu}_{k,m}')^{-1}(0)\setminus 0_{\R P^n}\to Y = \R_{>0} \times S^1 = \{(\rho,\alpha)\}.
\end{equation}

For $k>0$, $I$ acts on each spherical reduction fibre by $(x,y)\mapsto(-x,-y)$ and kills the antipodal monodromy, giving a trivial $(S^k\times S^m)/\Z_2$-bundle. For $k=0$, it instead translates the base angle by $1/2$ and leaves $y$ unchanged in the coordinates induced by $I_{0,m}$, whence quotienting still gives a trivial $S^m$-bundle. 

\begin{remark}
    \label{rk:double-cover}
    In the case $k \neq 0$, taking the double cover $\Pi$ is compatible with the symplectic quotient in the sense that $\pi = \pi' \circ \Pi$, where $\pi'$ is defined as in \eqref{eq:reductionTRPn}, and as in $\pi$ \eqref{eq:reductionTSn}. On the other hand, in the case $k=0$, the quotient maps $\pi,\pi'$ are related by a double cover. To be more precise, take the double cover $p(\rho,\alpha) = (\rho/2,2\alpha)$ on $Y$; then $p \circ \pi = \pi' \circ \Pi$. This explains why taking $\Pi(\cp^{\Ch, S}_{k,m})$ may not be embedded, depending on which curve in $Y$ it is the lift of. 
\end{remark}

\begin{proposition}
\label{prop:fiber bundle RPn}
The symplectic quotient defined in \cref{eq:reductionTRPn} has the following properties:
\begin{enumerate}
\item For~$k=0$, the map \eqref{eq:reductionTRPn} is a trivial~$S^m$-bundle over~$Y$;
\item For~$k \neq 0$, the map \eqref{eq:reductionTRPn} is a trivial~$(S^k \times S^m)/\Z_2$-bundle over~$Y$.
\end{enumerate}
\end{proposition}

\begin{definition}
    For $a>0$, choose an embedded contractible curve $\Gamma_{\RP}(a)$ bounding a disk of area $a$ in $Y$, and define the Chekanov Oakley--Usher Lagrangian by
    \[
        \cp^{\Ch}_{k,m}(a)=(\pi')^{-1}(\Gamma_{\RP}(a)) \subset T^*\RP^n.
    \]
\end{definition}

By the triviality of the bundle \eqref{eq:reductionTRPn}, we find that both~$\pkm$ and~$\pkmch$ are trivial bundles over~$S^1$ with fiber~$(S^k \times S^m)/\mathbb{Z}_2$ (or $S^m$, respectively) and therefore diffeomorphic. 

\begin{remark}
    \label{rk:preimage_covering}
    By \cref{rk:double-cover}, we find that $\Pi^{-1}(\cp^{(\Ch),\RP}_{k,m}) \cong \cp^{(\Ch),S}_{k,m}$ for all $k \neq 0$. When $k=0$, the same holds for $\cp^{\RP}_{k,m}$, but not for $\cp^{\Ch, \RP}_{k,m}$. The preimage of the latter under the covering is a disjoint union of two copies of $\cp^{\Ch,S}_{k,m}$. In all cases, the restriction of $\Pi$ to $\Pi^{-1}(\cp^{(\Ch),\RP}_{k,m})$ is a double cover. In the last case, it is the trivial double cover. 
\end{remark}

\begin{remark}
    \label{rk:displaceableTRPn}
    All tori $\pkmch \subset T^*\RP^n$ are displaceable, by lifting a displacement of the curve $\Gamma_{\RP}(a)$ in the quotient.
\end{remark}

\section{Comparison}
\label{sec:comparison}

In this section, we prove \cref{thm:compopen} comparing the Oakley--Usher Lagrangians to their Chekanov counterparts in the open manifolds $T^*S^n, T^*\RP^n$; and \cref{thm:compclosed} comparing them after compactification of the cotangent bundles to $Q^n, \CP^n$, respectively. \smallskip

\subsection{Proof of Theorem A}

The parameters~$r,a$ in~$\pkm(r)$ and~$\pkmch(a)$ are irrelevant in this proof, which is why we drop them from the notation. We first discuss the case of~$T^*S^n$. We start by determining for which~$k,m$, the Oakley--Usher type Lagrangians~$\pkm$ and~$\pkmch$ are diffeomorphic. For~$k=0$, the diffeomorphism type is~$S^1 \times S^m$ in both cases. For~$k \neq 0$, on the other hand, it is given by
\begin{equation}
    \pkm \cong (S^1 \times S^k \times S^m)/\Z_2,\qquad
    \pkmch \cong S^1 \times S^k \times S^m,
\end{equation}
where~$\Z_2$ acts by the antipodal map on each of the factors. As we have seen, the former can be viewed as a fibre bundle over~$S^1$ with fibre~$S^k \times S^m$ and monodromy~$\tau = (- \id_{S^k}) \times (- \id_{S^m })$. Note that~$-\id_{S^n}$ is isotopic to the identity if~$n$ is odd, which shows that~$\pkm$ and~$\pkmch$ are diffeomorphic when~$k,m$ are both odd. On the other hand, if at least one of~$k,m$ is even, we have~$\tau_* - \id \not\equiv 0$ on~$H_k(S^k \times S^m)$ (or~$H_m(S^k \times S^m)$, depending on which of~$k,m$ is even). 
From Proposition \ref{prop: homotopy distinction}, we conclude the first bullet point of the theorem.

Let us now determine when the Lagrangians are Lagrangian isotopic. It was shown in the proof of \cite[Proposition 5.1]{OakUsh16} that the minimal Maslov number of~$\pkm$ is~$k+m = n-1$ whenever~$k \neq 0$ and~$2m = 2(n-1)$ whenever~$k = 0$. On the other hand, we have proven, in the proof of \cref{prop:COUproperties}, that the minimal Maslov number of the Chekanov Oakley--Usher Lagrangian is~$2$ for all~$k,m$. Since the Maslov number is preserved by Lagrangian isotopies, this proves that~$\pkm$ and~$\pkmch$ are not Lagrangian isotopic whenever~$(k,m) \notin \{(0,1),(1,1)\}$.

Let us now show that the Lagrangian tori~$\cp_{0,1}, \cp^{\Ch}_{0,1} \subset T^*S^2$ are Lagrangian isotopic. Recall that both are contained in the level set
\begin{equation*}
    \overline{\mu}_{0,1}(q_1,q_2,q_3,p_1,p_2,p_3)
    = q_2p_3 - q_3p_2
    = 0.
\end{equation*}

In the plane~$\widehat{Y} = \overline{\mu}_{0,1}^{-1}(0) / S^1$ obtained as reduced space, they map to closed embedded curves~$\Upsilon = \pi(\cp_{0,1})$ and~$\Gamma = \pi(\cp^{\Ch}_{0,1})$, where~$\Upsilon$ encloses the point belonging to the lower-dimensional stratum (corresponding to~$0_{S^2}$) and~$\Gamma$ does not. The main idea is to look at the deformations~$\Upsilon_{\varepsilon}, \Gamma_{\varepsilon} \subset \widehat{Y}_{\varepsilon}$ in the reduced spaces~$\widehat{Y}_{\varepsilon} = \overline{\mu}_{0,1}^{-1}(\varepsilon) / S^1$. One can check that any~$\varepsilon \neq 0$ is a regular value of~$\overline{\mu}_{0,1}$ and hence~$\widehat{Y}_{\varepsilon}$ is a smooth standard symplectic plane. Hence~$\Upsilon_{\varepsilon}$ is isotopic to~$\Gamma_{\varepsilon}$ and thus their lifts in~$\overline{\mu}_{0,1}^{-1}(\varepsilon)$ are Lagrangian isotopic. We obtain the desired Lagrangian isotopy by 
\begin{enumerate}
    \item deforming~$\cp_{0,1} = \pi^{-1}(\Upsilon)$ to the lift of~$\Upsilon_{\varepsilon} \subset \widehat{Y}_{\varepsilon}$ by varying~$\varepsilon$;
    \item deforming the lift of~$\Upsilon_{\varepsilon}$ to the lift of~$\Gamma_{\varepsilon}$ as above;
    \item deforming the lift of~$\Gamma_{\varepsilon}$ to~$\cp^{\Ch}_{0,1} = \pi^{-1}(\Gamma)$ by varying~$\varepsilon$ to make it vanish.
\end{enumerate}
The same strategy works for the case~$(k,m) = (1,1)$ by varying~$\varepsilon$ in the level set of the reduction~$\overline{\mu}_{1,1}^{-1}(\varepsilon,\varepsilon)/(S^1 \times S^1)$. See also \cite[Remark~2.14]{OU2}.

\begin{remark} 
This strategy of proof only works in case~$G_{k,m} = SO(k+1) \times SO(m+1)$ is abelian, which is the case if and only if~$k,m \leqslant 1$. Indeed, if~$G_{k,m}$ is not abelian, then the level sets~$\overline{\mu}_{k,m}^{-1}(c)$ for~$c \neq 0$ are not invariant under~$G_{k,m}$, meaning that we cannot perform symplectic reduction.
\end{remark}

Since~$\pkm \subset T^*S^n$ is non-displaceable by \cite[Proposition 5.1]{OakUsh16} and~$\pkmch$ is displaceable by \cref{prop:COUproperties}, they are not Hamiltonian isotopic. This finishes the proof in the case of~$T^*S^n$.

We turn to the cotangent bundle of~$\R P^n$. As mentioned after Proposition \ref{prop:fiber bundle RPn},~$\pkm$ and~$\pkmch$ are diffeomorphic for all~$k$ and~$m$. 

We start with the case $k>0$. Suppose that $\cp^{\Ch,\RP}_{k,m}, \cp_{k,m}^{\RP} \subset T^*\RP^n$ are Lagrangian isotopic. Recall from \cref{rk:preimage_covering}, that $\Pi^{-1}(\cp^{(\Ch),\RP}_{k,m}) \cong \cp^{(\Ch),S}_{k,m}$ in case $k \neq 0$. In that case, we can thus lift the Lagrangian isotopy to a Lagrangian isotopy between $\cp^{\Ch,S}_{k,m}$ and $\cp_{k,m}^{S} \subset T^*\RP^n$. Indeed, this follows from the homotopy lifting property and from the fact that $\Pi$ is symplectic. By the first part of the theorem, this proves that this can only be the case if $k=m=1$. 

The Lagrangian isotopy between $\cp^{\Ch,S}_{1,1}, \cp_{1,1}^{S} \subset T^*S^3$ considered above is $T^2$-invariant. Since the involution $I$ on $T^*S^3$ is a $\Z_2$-subaction of the $T^2$-action, this Lagrangian isotopy factors through $\Pi$ and yields a Lagrangian isotopy between $\cp^{\Ch,\RP}_{1,1}, \cp_{1,1}^{\RP} \subset T^*\RP^3$. Since the former is displaceable by \cref{rk:displaceableTRPn} and the latter is not by \cite[Proposition~7.1]{OakUsh16}, we find that they are not Hamiltonian isotopic. 

Now let $k = 0$. We claim that $\cp^{\Ch,\RP}_{0,m}, \cp_{0,m}^{\RP} \subset T^*\RP^n$ are not even smoothly isotopic. The invariant distinguishing them is the image of 
\begin{equation}
    \label{eq:fundamental_induced}
    i_* \colon \pi_1(\cp^{(\Ch),\RP}_{0,m}) \rightarrow \pi_1(T^*\RP^n) \cong \Z_2,
\end{equation}
induced by the inclusion. We claim that for the Chekanov Oakley--Usher Lagrangian the image of this map is trivial. Let $\gamma$ be a loop in $\cp^{\Ch,\RP}_{0,m}$. Recall from \cref{rk:double-cover} that, in this particular case, the restriction
\begin{equation*}
    \Pi\vert_{\Pi^{-1}(\cp^{\Ch, \RP}_{0,m})} \colon \Pi^{-1}(\cp^{\Ch, \RP}_{0,m}) \rightarrow \cp^{\Ch, \RP}_{0,m} 
\end{equation*}
of the double cover is the trivial double cover. Therefore $\gamma$ lifts to a closed curve in the universal cover $\Pi$ and is thus contractible. Therefore the image of \eqref{eq:fundamental_induced} is trivial in the Chekanov case. Let us now prove that the image in the case of the conventional Oakley--Usher Lagrangian is non-trivial. To that end, consider the explicit path
\begin{equation*}
    c \colon [0,1] \rightarrow T^*S^n \quad
    c(t) = \left( \cos \pi t, y_0 \sin \pi t, -\textstyle\frac{s}{\pi} \sin \pi t, \textstyle\frac{s}{\pi}  y_0 \cos \pi t \right)
\end{equation*}
for a fixed $y_0 \in S^m$. The projection $\Pi \circ c$ is closed in $T^*\RP^n$ and generates $\pi_1(T^*\RP^n)$, since the endpoints $c(0),c(1)$ project to antipodes on $S^n$. Since $\Pi \circ c$ has image in $\cp^{\RP}_{0,m}$, this proves the claim.
\proofend

\subsection{Proof of Theorem B}

We start by discussing the case of~$Q^n$. The compactification by symplectic cut of the ambient space does not change the topology of the Lagrangians, therefore if~$k\neq 0$ and at least one of~$k,m$ is even,~$\cp_{k,m}(r)$ and~$\pkmch$ are distinguished by their integral homology groups. 

If~$a,r < \pi$ the Lagrangian isotopy of~$\cp_{1,1}^{\Ch}(r) \subset T^*S^3$ to~$\cp_{1,1}(a) \subset T^*S^3$ discussed in the above proof of \cref{thm:compopen} can be chosen such that it is contained in~$D^*S^3 \subset Q^3$, proving the third bullet point in the statement of \cref{thm:compclosed}.

Varying the size within~$(0,2\pi)$ gives a Lagrangian isotopy, so it suffices to establish the positive Lagrangian-isotopy assertion for sufficiently small sizes. Again we temporarily omit the parameters~$a,r$ as they do not play a role in the proof.

Now let us discuss the Maslov numbers in order to distinguish the Lagrangians up to Lagrangian isotopy and prove the second bullet point. A Lagrangian isotopy extends to an ambient smooth isotopy whose resulting map of pairs $\phi:(Q^n,\cp_{k,m}(r))\to(Q^n,\pkmch(a))$ induces an isomorphism
\[
\phi_*: \pi_2(Q^n,\cp_{k,m}) \xlongrightarrow{\cong} \pi_2(Q^n,\pkmch)
\]
preserving the Maslov class. Let~$L \in \{\cp_{k,m},\pkmch\}$.

By the proof of Theorem A, the Maslov homomorphisms
of $\cp_{k,m}$ and $\cp_{k,m}^{\mathrm{Ch}}$ in $T^*S^n$
take the values $n-1$ and $2$, respectively,
on suitable generators of their fundamental groups.
As~$\pi_2(Q^n) \cong \Z$, we are in the third case discussed in Subsection \ref{algebra section}. Recall that~$c_1(Q^n) = n$, and the Maslov class factors through a map 
\[
\overline{m}_L: \pi_1(L)\to \Z_{2n}, \qquad
\overline{m}_L(\gamma) :=m_L(\beta) \text{ mod } 2n,
\]
where~$\beta \in \pi_2(Q^n,L)$ such that~$\gamma = \partial \beta$. Assume first that~$k,m\geqslant 3$. This implies that~$\pi_2(L) =0$, and~$\pi_1(L) \cong \Z$. 
Since the isomorphism set is
\[
\Iso(\pi_1(\cp_{k,m}),\pi_1(\pkmch))\cong\Aut(\Z)=\{\pm1\},
\]
we obtain~$n-1 \equiv \pm 2 \quad \text{mod }2n$. As~$n\geqslant 7$, this equation has no solution. Therefore, no such Lagrangian isotopy can exist.

For the case~$(k,m)= (1,1)$, it can be shown in the same way as above for~$T^*S^n$ that the two Lagrangians (which have the topology of 3-tori) are Lagrangian isotopic in the quadric. We show below that the Maslov class cannot distinguish the two Lagrangians in the missing case~$k =1$,~$m\geqslant 3$ odd.

Let~$k=0$. Recall from the discussion surrounding \cref{def:YZZ} that the reduced space~$\widehat{Z'} = Q^n \sslash_{\overline{\mu}_{0,m} = 0} G_{0,m}$ is a stratified two-sphere, the smooth part~$Z'$ of which is symplectomorphic to the standard symplectic disk~$D = \{\pi \vert z \vert^2 < 2\pi\}$ of area~$2\pi$.

The OU-Lagrangian~$\cp_{0,m}(r)$ maps to a closed embedded curve bounding area~$2\pi - r$ in~$Z'$. The Chekanov OU-Lagrangian~$\cp^{\Ch}_{0,m}(a)$ maps to a closed embedded curve bounding area~$a$ in~$Z'$. This means that for~$a = 2\pi - r$, these curves are Hamiltonian isotopic in~$Z'$. The lift of this Hamiltonian isotopy to~$\overline{\mu}_{0,m}^{-1}(0) \subset Q^n$ proves the last bullet point concerning~$Q^n$. The fact that they are not Hamiltonian isotopic whenever $a \neq 2\pi - r$ follows from standard arguments using area and Maslov classes.

For~$\CP^n$, the proof is similar. Recall from \cref{def: projective YZ} that, for all~$k,m$, the smooth part of the reduced space~$\CP^n \sslash_{\overline{\mu}_{k,m} = 0}G_{k,m}$ is symplectomorphic to a symplectic disk with area~$\pi$. \proofend 

Note that \cref{thm:compclosed} does not treat the case where $k=1$, and $m \geqslant 3$ is odd. 
\begin{question}
    Let $m \geqslant 3$ be odd. Are $\cp_{1,m}, \mathcal{P}^{\textnormal{Ch}}_{1,m} \subset Q^n$ Lagrangian isotopic?
\end{question}
For completeness, we prove that the Maslov class cannot be used to distinguish these Lagrangians.

\begin{proposition}
\label{prop:k1modd}
Let~$m \geqslant 3$ be odd. Then the induced Maslov homomorphism on~$\pi_1$ does not distinguish~$\cp_{1,m}$ and~$\mathcal{P}^{\textnormal{Ch}}_{1,m}$ in~$Q^n$.
\end{proposition}
\proof
Since~$k=1, m\geqslant 3$, we have~$\pi_1(\cp_{1,m}) \cong \pi_1(\cp^{\text{Ch}}_{1,m}) \cong \mathbb{Z}^2$. Choose a basis~$(\gamma_1,\gamma_2)$ of~$\pi_1(\cp_{1,m})$ and a basis~$(\widetilde{\gamma}_1,\widetilde{\gamma}_2)$ of~$\pi_1(\cp^{\text{Ch}}_{1,m})$ such that
\[
\overline m_{\cp_{1,m}}(\gamma_1)=n-1 \text{ mod } 2n, \qquad \overline m_{\cp_{1,m}}(\gamma_2)=0 \text{ mod } 2n.
\]
\[
\overline m_{\cp^{\text{Ch}}_{1,m}}(\widetilde{\gamma}_1)=2 \text{ mod } 2n, \qquad \overline m_{\cp^{\text{Ch}}_{1,m}}(\widetilde{\gamma}_2)=0 \text{ mod } 2n.
\]
Note that
\[
\Iso(\pi_1(\cp_{1,m}),\pi_1(\cp^{\text{Ch}}_{1,m}))\cong\GL(2;\Z).
\]
Set 
\[
a:=\frac{n-1}{2}, \qquad b:= n.
\]
Since~$n=k+m+1=m+2$ is odd,~$a\in \Z$, and moreover
\[
\gcd(a,b) = \gcd\left(\frac{n-1}{2},n \right) = 1.
\]
By Bézout's identity, there exist~$c,d\in \Z$ such that~$ad-bc = \pm 1$. Therefore, 
\[
A := \begin{pmatrix}
        a & b  \\
        c & d  
        \end{pmatrix}\in \GL(2;\Z).
\]
We obtain 
\[
(\overline m_{\cp^{\text{Ch}}_{1,m}} \circ A)(\gamma_1) = \overline m_{\cp^{\text{Ch}}_{1,m}}(a \widetilde{\gamma}_1+c\widetilde{\gamma}_2) \equiv  2a = n-1 = \overline m_{\cp_{1,m}}(\gamma_1).
\]
\[
(\overline m_{\cp^{\text{Ch}}_{1,m}} \circ A)(\gamma_2) = \overline m_{\cp^{\text{Ch}}_{1,m}}(b \widetilde{\gamma}_1+d\widetilde{\gamma}_2) \equiv  2b = 2n \equiv 0 =\overline m_{\cp_{1,m}}(\gamma_2).
\]
Therefore, 
\[
\overline m_{\cp^{\text{Ch}}_{1,m}} \circ A = \overline m_{\cp_{1,m}},
\]
showing that the Maslov homomorphisms are equivalent under the natural~$\GL(2;\Z)$-action on~$\Hom(\Z^2,\Z_{2n})$. Therefore the Maslov class does not distinguish the two Lagrangians in this case. \proofend

\bibliographystyle{abbrv}
\bibliography{OU1}

\end{document}